\documentclass[12pt,reqno] {amsart}
\usepackage{amsfonts,amssymb,amsmath,amsthm}
\usepackage{mathrsfs}
\usepackage{url}
\usepackage{enumerate}
\usepackage{xcolor}
\usepackage{todonotes}

\newtheorem{theorem}{Theorem}[section]
\newtheorem*{theorem*}{Theorem}
\newtheorem{lemma}[theorem]{Lemma}

\newtheorem{prop}[theorem]{Proposition}
\newtheorem{cor}[theorem]{Corollary}
\theoremstyle{definition}

\newtheorem{rem}[theorem]{Remark}

\newtheorem{Question}[theorem]{Question}

\author[S. Ghara]{Soumitra Ghara}
\author[R. Gupta]{Rajeev Gupta}
\author[M. R. Reza]{Md. Ramiz Reza}

\address[S. Ghara]{Department of Mathematics and Statistics\\
Indian Institute of Technology Kanpur, India
}

\address[R. Gupta]{School of Mathematics and Computer Science\\
Indian Institute of Technology Goa, India}

\address[Md. R. Reza]{Department of Mathematics and Statistics\\
Indian Institute of Technology Kanpur, India
 }

\email[S. Ghara]{ghara90@gmail.com}
\email[R. Gupta]{rajeev@iitgoa.ac.in}
\email[M. R. Reza]{ramiz.md@gmail.com}

\thanks{The first author was supported by the SERB through the NPDF post-doctoral fellowship (Ref.No. PDF/2019/002724). The second author was supported through the INSPIRE faculty grant (Ref. No. DST/INSPIRE/04/2017/002367) and the third author was supported by the institute post-doctoral fellowship of IIT Kanpur.}


\begin{document}
\title [Analytic $m$-isometries and weighted Dirichlet-type spaces]{Analytic $m$-isometries and weighted Dirichlet-type spaces}
\begin{abstract}
Corresponding to any $(m-1)$-tuple of semi-spectral measures on the unit circle, a weighted Dirichlet-type space is introduced and studied. We prove that 
every analytic $m$-isometry which satisfies a certain set of operator inequalities can be represented as the  operator of  multiplication by the coordinate function on such a weighted Dirichlet-type space. This extends a result of Richter as well as of Olofsson on analytic $2$-isometries. 
We also prove that all left invertible
$m$-concave operators satisfying the aforementioned operator inequalities admit a Wold-type decomposition.
This result serves as a key ingredient in our model theorem and it also generalizes a result of Shimorin on a class of $3$-concave operators.
\end{abstract}


\maketitle

\section{Introduction}
In what follows, $\mathcal H$ and $\mathcal E$ will denote  complex separable Hilbert spaces. The notations $\mathcal B(\mathcal H)$ and $\mathcal B(\mathcal E)$ will stand for the set of all bounded linear operators on $\mathcal H$ and $\mathcal E$ respectively. Let $m$ be a positive integer. An operator $T$ in $\mathcal B(\mathcal H)$ is said to be an {\it $m$-isometry} (resp. {\it $m$-concave}) if $\beta_m(T) = 0$ (resp. $\beta_m(T) \leqslant 0$), where the {\it $m$-th 
defect operator} $\beta_m(T)$ is defined by
\begin{eqnarray*}
\beta_m(T):= \sum\limits_{j=0}^m \binom{m}{j}(-1)^{m-j}{T^*}^jT^j.
\end{eqnarray*}
Further, we set $\beta_0(T)=I.$ An operator $T$ in $\mathcal B(\mathcal H)$ is said to be {\it expansive} ({\it norm-increasing}) if $\beta_1(T)\geqslant 0.$ The notion of $m$-isometric operators was introduced and studied by Agler in \cite{Agler}. Later in a series of papers, Agler and Stankus studied the class of $m$-isometries extensively, see \cite{AglerStan1, AglerStan2, AglerStan3}. An operator $T$ in $\mathcal B(\mathcal H)$ is called {\it analytic} if $\mathcal H_{\infty}(T)= \{0\},$ where $\mathcal H_{\infty}(T),$ the {\it hyper-range} of $T,$ is defined to be $\mathcal H_{\infty}(T) := \bigcap_{n=0}^{\infty} T^n(\mathcal H).$ 

Note that a $1$-isometry is simply an isometry on a Hilbert space. 
The classical Wold-Kolmogorov decomposition theorem says that an isometry can be uniquely written as the direct sum of a unitary operator and an analytic isometry, see \cite[Theorem 1.1]{NF}. More generally, it follows from \cite[Proposition 3.4]{Shimorin} that every expansive $m$-isometry can be uniquely written as the direct sum of a unitary operator and an  analytic $m$-isometry. Thus the study of analytic $m$-isometries are natural starting point in order to investigate expansive $m$-isometries.

Let $H^2(\mathcal E)$ denote the Hardy space of $\mathcal E$-valued holomorphic functions on the open unit disc $\mathbb D$, that is,
\begin{align*}
H^2(\mathcal E) = \Big\{\textstyle\sum\limits_{j=0}^{\infty}a_jz^j : a_j \in\mathcal E, \sum\limits_{j=0}^{\infty}\|a_j\|^2_{_{\mathcal E}} < \infty  \Big\}. 
\end{align*} 
The operator $M_z$ of \emph{multiplication by the coordinate function} on $H^2(\mathcal E)$ is an analytic isometry. 
From the  Wold-Kolmogorov decomposition theorem, it follows that every analytic isometry $T$ in $\mathcal B(\mathcal H)$ is unitarily equivalent to the operator $M_z$ on the Hardy space $H^2(\mathcal E)$ with $\mathcal E = \ker T^*$ (cf. \cite[Theorem 1.1]{NF}). 
This is known as the model theorem for analytic isometries.

In 1991, Richter \cite[Theorem 3.7, Theorem 5.1]{R} proved that any cyclic analytic $2$-isometry is unitarily equivalent to the operator of multiplication by the coordinate function on a {\emph{Dirichlet-type space $D(\mu)$}} for some finite positive Borel measure $\mu$ on the \emph{unit circle} $\mathbb T,$ where 
$$D(\mu):= \left\{f\in \mathcal O(\mathbb D) : \int_{\mathbb D}|f^{\prime}(z)|^2 P_{\mu}(z) \,dA(z) < \infty\right\},$$
$P_{\mu}(z)$ denotes the \emph{Poisson integral} of the measure $\mu,$ $dA(z)$ denotes the \emph{normalized Lebesgue area measure} on the \emph{open unit disc} $\mathbb D$  and $\mathcal O(\mathbb D)$ denotes the space of complex valued holomorphic functions on the open unit disc $\mathbb D.$ 
Olofsson generalized this result of Richter and obtained a model for an arbitrary analytic $2$-isometry by considering weighted Dirichlet-type spaces $D(\mu)$ associated to a positive $\mathcal B(\mathcal E)$-valued operator measure $\mu$ on $\mathbb T,$ see \cite[Theorem 3.1 and 4.1]{OLLOF2iso}.

Recently, Rydhe characterized the class of cyclic $m$-isometries in terms of shifts on abstract spaces of weighted Dirichlet-type. It was shown that every cyclic $m$-isometry is unitarily equivalent to the multiplication operator by the coordinate function on a Hilbert space $\mathcal D^2_{\overrightarrow{\mu}}$  (depending on the operator as well as on the cyclic vector chosen), induced by an \emph{allowable} $m$-tuple $\overrightarrow{\mu}= (\mu_0,\ldots,\mu_{m-1})$ of distributions on the unit circle $\mathbb T$, see \cite[Theorem 3.1]{Rydhe}. 
Motivated by Rydhe's model, in this article, we attempt to find Richter-type model as well as Olofsson-type model for a class of analytic $m$-isometries. In case of $m=2,$ this class coincides with the class of all analytic $2$-isometries.

Let $\mathcal E$ be a Hilbert space and $\mathcal O(\mathbb D,\mathcal E)$ denote the space of $\mathcal E$-valued holomorphic functions on $\mathbb D.$ By a $\mathcal B(\mathcal E)$-valued \emph{semi-spectral measure} on the unit circle $\mathbb T,$ we mean a positive $\mathcal B(\mathcal E)$-valued operator measure on $\mathbb T$, see \cite[p. 719]{OLLOF2iso} for definition and basic properties. The notation $\mathcal M_{+}(\mathbb T,\mathcal B(\mathcal E))$ stands for the set of all $\mathcal B(\mathcal E)$-valued \emph{semi-spectral measures} on $\mathbb T$. We will simply write semi-spectral measure in place of $\mathcal B(\mathcal E)$-valued \emph{semi-spectral measure} whenever there is no possibility of ambiguity about the underlying space $\mathcal E$. The set $\mathcal M_+(\mathbb T, \mathcal B(\mathbb C))$ coincides with the set of all finite positive Borel measures on $\mathbb T$ and is abbreviated to $\mathcal M_+(\mathbb T).$ For a  $\mu\in \mathcal M_{+}(\mathbb T,\mathcal B(\mathcal E)),$ $f\in \mathcal O(\mathbb D,\mathcal E)$ and  a positive integer $n,$ we consider $D_{\mu,n}(f),$ the {\it weighted Dirichlet integral of order $n$,}  defined by 
\begin{align*}
D_{\mu,n}(f):= \frac{1}{n! (n-1)!}\int_{\mathbb D} \big\langle P_{\!\mu}(z)f^{(n)}(z), f^{(n)}(z)\big\rangle (1-|z|^2)^{n-1}dA(z),
\end{align*}
where $f^{(n)}$ denotes the $n$-th order complex derivative of the function $f,$ that is, $f^{(n)}(z):= \frac{d^n}{dz^n}f(z),$ and $P_{\mu}(z)$ denotes the \emph{Poisson integral} of  the semi-spectral measure $\mu$. In the case of $\mathcal E= \mathbb C,$ the integral $D_{\mu,n}(\cdot)$ considered here coincides with the notion of {weighted Dirichlet integral of order $n$} as described in \cite{Rydhe}. The novelty here lies in considering the associated semi-inner product space $\mathcal H_{\mu,n}(\mathcal E)$ and the Hilbert space $\mathcal H_{\pmb\mu}(\mathcal E)$ defined below.

For any $\mathcal B(\mathcal E)$-valued semi-spectral measure $\mu$ on $\mathbb T$ and any positive integer $n,$ we associate a natural linear space $$\mathcal H_{\mu,n}(\mathcal E) := \{f\in \mathcal O(\mathbb D,\mathcal E): D_{\mu,n}(f) < \infty\}.$$
It is routine to verify that $\mathcal H_{\mu,n}(\mathcal E)$  is a semi-inner product space with respect to the semi-inner product 
\begin{align*}
\big\langle f, g\big\rangle = \frac{1}{n! (n-1)!}\int_{\mathbb D} \big\langle P_{\!\mu}(z)f^{(n)}(z), g^{(n)}(z)\big\rangle (1-|z|^2)^{n-1}dA(z),\,\,\,\,\,f , g\in \mathcal H_{\mu,n}(\mathcal E).
\end{align*}
The space $\mathcal H_{\mu,1}(\mathbb C),$ that is, the case of $n=1$ and $\mathcal E=\mathbb C,$ has been studied extensively in the literature, see for instance \cite{R, RS, Aleman, Sara}, \cite[Ch.7]{Primer} and references therein. The properties of the space $\mathcal H_{\mu,1}(\mathcal E)$ has been discussed in \cite{OLLOF2iso}. In Section 2 of this article, we explore various properties of the higher order weighted Dirichlet integrals $D_{\mu,n}(\cdot)$ and the associated spaces $\mathcal H_{\mu,n}(\mathcal E).$

One of the interesting properties of the functions in $\mathcal H_{\mu,n}(\mathcal E)$ is that these can be approximated  through $r$-dilations. For $0 < r< 1,$  the $r$-dilation of a function  $f\in\mathcal O(\mathbb D,\mathcal E)$ is denoted by $f_r$ and is defined by $f_r(z) := f(rz),\,z\in\mathbb D.$  
In Section 3 of this article, we show that $D_{\mu,n}( f_r-f)\to 0$ as $r\to 1,$ for every $f\in \mathcal H_{\mu,n}(\mathcal E).$ In case $\mathcal E=\mathbb C$ and $\mu$ is a finite  positive Borel measure in $\mathcal M_+(\mathbb T),$ Richter and Sundberg established the approximation through $r$-dilations 
by showing that, for every $f\in\mathcal H_{\mu,1}(\mathbb C)$ and $0 < r < 1,$ $D_{\mu,1}( f_r)\leqslant C D_{\mu,1}(f)$ with constant $C=4,$ see \cite[Theorem 5.2]{RS}.  Aleman improved this result by replacing the constant $4$ by $5/2$ (cf. \cite[Lemma 4.1]{Aleman}). This remained the best value known for the constant until Sarason, in \cite[Proposition 3]{Sara}, proved  that $D_{\mu,1}(f_r) \leqslant D_{\mu,1}(f)$
for every $f\in \mathcal H_{\mu,1}(\mathbb C),$ and for every $0 <r <1,$ see also \cite[Lemma 7.3.2]{Primer}.  In Section 3 using Theorem \ref{Radial domination}, we show that the result of Sarason remains true even for the functions in $\mathcal H_{\mu,n}(\mathcal E),$ where $n$ is an arbitrary positive integer and $\mathcal E$ is an arbitrary complex separable Hilbert space.  

\begin{theorem}\label{Radial domination}
Let $\mathcal E$ be a complex separable Hilbert space and $\mu \in \mathcal M_+(\mathbb T,\mathcal B(\mathcal E)).$ Then for every positive integer $n$ and  $f\in \mathcal H_{\mu,n}(\mathcal E),$  we have 
\begin{align*}
D_{\mu,n}(f_r) \leqslant  D_{\mu,n}(f),\,\,\,\,\,\,\,\, 0 < r< 1,  
\end{align*}
 and consequently $D_{\mu,n}(f_r-f)\to 0$ as $r\to 1.$
\end{theorem}
 In Section 4, for $m\geqslant 2,$ we consider $(m-1)$-tuple $\pmb\mu= (\mu_1,\ldots,\mu_{m-1})$ of $\mathcal B(\mathcal E)$-valued semi-spectral measures on $\mathbb T$ and then we introduce the corresponding \emph{weighted Dirichlet-type space} $\mathcal H_{\pmb \mu}(\mathcal E)$ of $\mathcal E$-valued holomorphic functions in the following way : 
\begin{align*}
\mathcal H_{\pmb\mu} (\mathcal E):= \bigcap\limits_{j=1}^{m-1} \mathcal H_{\mu_j,j}(\mathcal E)\bigcap H^2({\mathcal E})
\end{align*}
with the norm $\|\cdot\|_{\pmb \mu}$ given by
\begin{align*}
\|f\|^2_{\pmb \mu}= \|f\|^2_{\!_{H^2(\mathcal E)}}+ \sum\limits_{j=1}^{m-1} D_{\mu_j,j}(f),\quad f\in \mathcal H_{\pmb \mu}(\mathcal E).
\end{align*}
The space $\mathcal H_{\pmb \mu}(\mathcal E)$ with respect to the norm $\|\cdot\|_{\pmb\mu}$ is shown to be a reproducing kernel Hilbert space. 
We also show that the multiplication operator $M_z$ on the Hilbert space $\mathcal H_{\pmb \mu}(\mathcal E)$ is a bounded analytic $m$-isometry, see Theorem \ref{bounded m iso vector valued}. It is easily noted that the set of $\mathcal E$-valued polynomials is contained in $\mathcal H_{\pmb \mu}(\mathcal E).$ It is therefore natural to ask:
For $m\geqslant 2$ and  an $(m-1)$-tuple $\pmb \mu= (\mu_1,\ldots,\mu_{m-1})$ of $\mathcal B(\mathcal E)$-valued semi-spectral measures on $\mathbb T,$ is the set of $\mathcal E$-valued polynomials dense in $\mathcal H_{\pmb \mu}(\mathcal E)?$
In case $m=2,$ the Hilbert space $\mathcal H_{\pmb \mu}(\mathcal E)$ coincides with the model space $D(\mu_1),$ described by Olofsson in \cite{OLLOF2iso}, and an affirmative answer to the question is known in this case, see \cite[Corollary 3.1]{OLLOF2iso}.
Using Theorem \ref{Radial domination}, we answer this question in affirmative in general, see Proposition \ref{polydense through dilation}. 
In Theorem \ref{defect formula for model space vector valued}, we show that the operator $M_z$ on $\mathcal H_{\pmb \mu}(\mathcal E)$ is an analytic $m$-isometry satisfying the following set of operator inequalities: 
\begin{align*}
\beta_r(M_z)\geqslant \sum_{n=1}^{\infty} {L_{M_z}^*}^{\!\!\!\!\!n}\,\beta_{r+1}({M_z})L_{M_z}^n,\quad r=1,\ldots,m-2,
\end{align*} 
where for a left invertible operator $T$ in $\mathcal B(\mathcal H),$ we use the notation $L_T$ to denote the left inverse $(T^*T)^{-1}T^*$ of $T$. This set of operator inequalities turns out to be a characteristic property of an analytic $m$-isometry to represent it as the multiplication operator $M_z$ on $\mathcal H_{\pmb \mu}(\mathcal E)$ for some $(m-1)$-tuple $\pmb \mu = (\mu_1,\ldots,\mu_{m-1})$ of $\mathcal B(\mathcal E)$-valued semi-spectral measures on $\mathbb T.$ In Section 6, we establish the following model theorem which is one of the key findings of this article .
\begin{theorem}\label{model thm}
Let $m\geqslant 2$ and $T\in \mathcal B(\mathcal H).$ Then $T$ is an analytic $m$-isometry satisfying 
\begin{align}\label{mu-positive}
\beta_r(T)\geqslant \displaystyle\sum_{k=1}^{\infty} {L_T^*}^k\beta_{r+1}(T) L_T^k, \quad r=1,\ldots, m-2 
\end{align}
if and only if $T$ is unitarily equivalent to the multiplication operator $M_z$ on $\mathcal H_{\pmb \mu}(\mathcal E)$ for some $(m-1)$-tuple $\pmb \mu =(\mu_1,\ldots,\mu_{m-1})$  of $\mathcal B(\mathcal E)$-valued semi-spectral measures on $\mathbb T$ with $\mathcal E= \ker T^*.$  
\end{theorem}

Note that for $m=2,$ the operator inequalities in \eqref{mu-positive} are vacuously satisfied.
In case $\ker T^*$ is of dimension one in above theorem, the associated $(m-1)$-tuple of measures is uniquely determined. In the general case, the associated $(m-1)$-tuple of semi-spectral measures is determined upto conjugation by a unitary operator, see Proposition \ref{Uniqueness} for the details.

For the operator $M_z$ on $\mathcal H_{\pmb \mu}(\mathcal E),$ we note that $\ker  M_z^* = \mathcal E.$ Since the set of all $\mathcal E$-valued polynomials is dense in $\mathcal H_{\pmb \mu}(\mathcal E),$ it follows that the operator $M_z$ on $\mathcal H_{\pmb \mu}(\mathcal E)$ has the \emph{wandering subspace property}. Following \cite[Definition 2.4]{Shimorin}, we say that an operator $T$  in $\mathcal B(\mathcal H)$ has the wandering subspace property if it satisfies 
\begin{align*}
\bigvee \big\{T^nx:x\in\ker T^*, \, n\geqslant 0\big\} = \mathcal H.
\end{align*}
The term ``wandering subspace" is attributed to Halmos \cite{HAL}. 
In the literature, this property has been studied extensively, for instance see  \cite{Rinv, ARS,Shimorin,OLW,IZU1,
IZU2,IZU3,AST} and references therein. In view of Theorem \ref{model thm}, it is necessary that an analytic $m$-isometry $T$ with $m\geqslant 2,$ which satisfies the operator inequalities in  \eqref{mu-positive}, must have the wandering subspace property. In Section 5 of this article, we show that all analytic left invertible $m$-concave operators which satisfy the operator inequalities in \eqref{mu-positive} have the  wandering subspace property.
This plays a crucial role in proving Theorem \ref{model thm}.
\begin{theorem}\label{Wold-type-mconcave}
Let $T$ be an analytic left invertible $m$-concave operator in $\mathcal B(\mathcal H)$ for some $m \geqslant 2.$ If $T$ satisfies
\begin{eqnarray*}
\beta_r(T)\geqslant \sum_{k=1}^{\infty} {L_T^*}^k\beta_{r+1}(T) L_T^k, \quad r=1,\ldots, m-2, 
\end{eqnarray*}
then $T$ has the wandering subspace property.
\end{theorem}

As a consequence of this theorem, it is obtained in Theorem \ref{mconcave new} that any left invertible $m$-concave operator satisfying \eqref{mu-positive} admits a Wold-type decomposition (see Section \ref{sec5} for definition).
For $m=3,$ we find that Theorem \ref{Wold-type-mconcave} improves the result in \cite[Corollary 3.10]{Shimorin}, see Remark \ref{remark strct containment vector valued} for details. 
Note that for $m=2,$ 
Theorem \ref{Wold-type-mconcave} gives us the wandering subspace property for every analytic $2$-concave operator, see \cite[Theorem 1]{Rinv}. Since both, the analyticity and the concavity property, of an operator $T$ in $\mathcal B(\mathcal H)$ remain unchanged under the restriction of $T$ to its invariant subspaces, it follows that for any $\mathcal B(\mathcal E)$-valued semi-spectral measure $\mu$ on $\mathbb T,$ $M_z|_{\mathcal W}$ has the wandering subspace property for every $M_z$ invariant subspace $\mathcal W$ of $\mathcal H_{\mu}(\mathcal E).$
At this point, it is natural to ask if $\pmb \mu = (\mu_1,\ldots,\mu_{m-1})$ is an $(m-1)$ tuple of $\mathcal B(\mathcal E)$-valued semi-spectral measures on $\mathbb T$ with $m\geqslant 3,$ whether $M_z|_{\mathcal W}$ has the wandering subspace property or not, for every $M_z$ invariant subspace $\mathcal W$ of $\mathcal H_{\pmb \mu}(\mathcal E).$ Unfortunately, we are not able to answer this question right now as the conditions \eqref{mu-positive} in Theorem \ref{Wold-type-mconcave} may not be stable under the restriction of $T$ to its invariant subspaces. We hope to discover the answer to this question in our future work.

\section{Properties of Weighted Dirichlet integral of order $n$}

The symbols $\mathbb Z, \mathbb N$ and $\mathbb Z_+$ will denote the set of all integers, positive integers and non-negative integers respectively. By a $\mathcal B(\mathcal E)$-valued {\it semi-spectral measure} $\mu$ on $\mathbb T,$ we mean a finitely additive set function from the Borel $\sigma$-algebra of $\mathbb T$  into the set of all positive operators in $\mathcal B(\mathcal E)$ such that $\mu_{x,y}(\cdot):=\langle{\mu(\cdot)x}, {y}\rangle$  defines a regular complex Borel measure on $ \mathbb T$ for every $x, y \in \mathcal E.$ 
As in the previous section, the notation $\mathcal M_+(\mathbb T, \mathcal B(\mathcal E))$ stands for the set of all $\mathcal B(\mathcal E)$-valued semi-spectral measures on $\mathbb T.$
For $\mu \in \mathcal M_+(\mathbb T, \mathcal B(\mathcal E)),$
consider the {\it Poisson integral} $P_{\!\mu}$ of $\mu$ given by
\begin{align*}
P_{\!\mu}(z):=\int_{ \mathbb T} P(z, \zeta) d\mu(\zeta), \quad z \in \mathbb D,
\end{align*}
where $P(z,\zeta)= \frac{1-|z|^2}{|z-\zeta|^2},$ $z\in\mathbb D,$ $\zeta\in\mathbb T$ is the {\it Poisson kernel} for the open unit disc $\mathbb D.$  Note that for each $z\in \mathbb D,$ $P_{\!\mu}(z)$ is a positive operator in $\mathcal B(\mathcal E)$  
and for every $x,y\in\mathcal E,$ the function $\langle P_{\!\mu}(z)x,y\rangle$ is the {\it Poisson integral} of the complex Borel measure $\mu_{x, y}$ on the unit circle $\mathbb T.$ Thus the map $z\mapsto \langle P_{\!\mu}(z)x,y\rangle$ is a complex valued harmonic function on $\mathbb D.$  
For a positive integer $n,$ consider the {\textit{weighted Dirichlet integral }}$D_{\mu,n}(f)$ {\textit{of order n}},  for $f\in\mathcal O(\mathbb D,\mathcal E)$ defined by
\begin{align*}
D_{\mu,n}(f)= 
\frac{1}{n!(n-1)!}\displaystyle\int_{\mathbb D}\big\langle P_{\!\mu}(z) f^{(n)}(z),f^{(n)}(z)\big\rangle (1-|z|^2)^{n-1}dA(z).
\end{align*} 
It will be also useful to consider the weighted Dirichlet integral $D_{\mu,0}(f)$ of order $0,$ given by 
\begin{align*}
D_{\mu,0}(f) = \lim_{R\to 1}\displaystyle\int_{\mathbb T}\big\langle P_{\!\mu}(R\zeta)f(R\zeta),f(R\zeta)\big\rangle d\sigma(\zeta),
\end{align*}
provided the limit exists, where $d\sigma$ denotes the {\it normalized arc length measure} on the unit circle $\mathbb T.$  We will see in Corollary \ref{Containment of Hmu vector valued} that $D_{\mu,0}(f) < \infty,$ that is, the corresponding limit exists and is finite, whenever $D_{\mu,n}(f) < \infty$  for some $n\in\mathbb N.$


The following lemma, a straightforward generalization of \cite[Lemma 3.2 and Proposition 3.4]{Rydhe} from the case of $\mathcal E= \mathbb C$ to an arbitrary complex separable Hilbert space $\mathcal E,$ provides a formula for computing  $D_{\mu,n}(f)$ for an arbitrary function $f$ in $\mathcal O(\overline{\mathbb D},\mathcal E),$ the space of all $\mathcal E$-valued functions holomorphic in some neighbourhood of the closed unit disc $\overline{\mathbb D}.$ This can be obtained by computing $D_{\mu,n}(f)$ for any $\mathcal E$-valued polynomial $f$ and then using the observation that the power series of every function $f$ in $\mathcal O(\overline{\mathbb D},\mathcal E)$ converges uniformly on the closed unit disc $\overline{\mathbb D},$ one gets the desired formula for every  $f$ in  $\mathcal O(\overline{\mathbb D},\mathcal E),$ cf.  \cite[Equation(3.2)]{OLLOF2iso}.  
The second part of the lemma gives a relationship among the weighted Dirichlet integrals $D_{\mu,n}(f)$ for every $f$ in $\mathcal O(\overline{\mathbb D},\mathcal E),$ which can be easily verified from the formula obtained for $D_{\mu,n}(\cdot)$ in the first part. 
\begin{lemma}\label{Rydhe diff vector valued}
Let $f\in\mathcal O(\overline{\mathbb D},\mathcal E)$ and $\mu$ be a $\mathcal B(\mathcal E)$-valued semi-spectral  measure on the unit circle $\mathbb T.$
\begin{itemize}
\item[\rm (i)]If $f(z)= \displaystyle\sum_{j=0}^{\infty}\hat{f}(j) z^j$ for $z\in\overline{\mathbb D}$ then
\begin{align*}
D_{\mu,n}(f)= \sum\limits_{k,l=n}^{\infty} \binom{k\wedge l}{n}\big\langle \hat{\mu}(l-k)\hat{f}(k),\hat{f}(l)\big\rangle,~~~~n\in \mathbb Z_+,
\end{align*} where $k\wedge l=\min\{k,l\},$  and $\hat{\mu}(j)= \displaystyle\int_{\mathbb T} \zeta^{-j} d\mu(\zeta)$ for every $j\in\mathbb Z.$ Moreover, the series on the right hand side is absolutely convergent.
\item[\rm (ii)]
The following identity holds:

\begin{align*}
D_{\mu,n+1}(zf)- D_{\mu,n+1}(f) = D_{\mu,n}(f),~~~~~~~~n\in \mathbb Z_+.
\end{align*}
\end{itemize}
\end{lemma}
In Proposition \ref{R diff vetor valued}, it is noted that the following refined version of the identity in Lemma \ref{Rydhe diff vector valued}(ii) holds for functions in $\mathcal O(\mathbb D,\mathcal E):$ 
\begin{align*}
D(\mu, n+1, R, zf) - R^2D(\mu, n+1, R, f)=R^2 D(\mu, n, R, f)
\end{align*}
where for $0<R<1$ and for any $f\in\mathcal O(\mathbb D,\mathcal E),$ the refined weighted Dirichlet integral of order $n$ is defined by
\begin{align*}
D(\mu, n, R, f):= 
\begin{cases}
\frac{1}{n!(n-1)!}\displaystyle\int_{R\mathbb D}\big\langle P_{\!\mu} (z)f^{(n)}(z), f^{(n)}(z)\big\rangle  (R^2-|z|^2)^{n-1}  ~dA(z),\quad & n\geqslant 1,\\[10pt]
\displaystyle\int_{\mathbb T}\big\langle  P_{\!\mu} (R\zeta)f(R\zeta), f(R\zeta)\big\rangle  d\sigma(\zeta), & n=0.
\end{cases}
\end{align*}
Before we come to the refined version of the concerned identity, in the following proposition, we show that the refined weighted Dirichlet integral $D(\mu, n, R, f)$ is equal to the weighted Dirichlet integral of order $n$ of the function $f_{\!_R},$ the $R$-dilation of $f,$ given by $f_{\!_R}(z)=f(Rz),\,z\in\mathbb D,$ with respect to an appropriate measure.
\begin{prop}\label{MCT of integral}
Let $\mu\in \mathcal M_+(\mathbb T,\mathcal B(\mathcal E))$  and $0 < R < 1.$  Consider the $\mathcal B(\mathcal E)$-valued semi-spectral measure $\lambda _{R}$ on $\mathbb T$ defined by $d\lambda_{R}(\zeta)= P_{\!\mu}(R\zeta)d\sigma(\zeta).$ Then for every $f\in \mathcal O(\mathbb D,\mathcal E)$ and $n\in\mathbb Z_+,$ it follows that 
\begin{align*}
D(\mu, n, R, f) = D_{\lambda _{R},n}\,(f_{\!_R}).
\end{align*}
Furthermore, for every $f\in \mathcal O(\mathbb D,\mathcal E)$ and $n\in\mathbb N,$ it follows that  
$D(\mu, n, R, f)$ increases to $D_{\mu,n}(f)$ as $R$ increases to $1.$
\end{prop}
\begin{proof}
For any $x\in\mathcal E$, note that $({\lambda_R})_{x,x}$ is a finite positive Borel measure on $\mathbb T$ and is given by 
$d({\lambda_R})_{x,x}(\zeta)= \langle P_{\!\mu}(R\zeta)x, x \rangle d\sigma(\zeta).$ Since the function $z\mapsto \langle P_{\!\mu}(Rz)x,x\rangle$ is  a harmonic function on a neighborhood of the closed unit disc $\overline{\mathbb D}$, it follows that 
\begin{equation*}
P_{\!{(\lambda_R)}_{x,x}}(z)= \langle P_{\!\mu}(Rz)x,x\rangle,\quad z\in\mathbb D.
\end{equation*}
 Also, since $P_{\!{(\lambda_R)}_{x,x}}(z)=\langle P_{\!\lambda_R}(z)x,x\rangle$  for every $z\in\mathbb D$ and $x\in\mathcal E$, it follows from the above equality that 
\begin{eqnarray}\label{eqnprop3.3 1}
P_{\!\lambda_R}(z)=P_{\!\mu}(Rz),\quad z\in\mathbb D.
\end{eqnarray}
Note that $f_{\!_R}^{(n)}(w)= R^{n} f^{(n)}(Rw)$ for every $w\in\mathbb D$ and for all $n\in\mathbb Z_+.$  For $n\geqslant 1,$ by a change of variables together with \eqref{eqnprop3.3 1}, we get that
\begin{eqnarray*}
D_{\lambda_R, n}\,(f_{\!_R})
&=& \int_{\mathbb D} \big\langle P_{\!\lambda_R}(w) f_{\!_R}^{(n)}(w), f_{\!_R}^{(n)}(w)\big\rangle (1-|w|^2)^{n-1}dA(w)\\
&=& \int_{R \mathbb D} \big\langle P_{\!\mu}(z)f^{(n)}(z), f^{(n)}(z) \big\rangle (R^2-|z|^2)^{n-1}dA(z)\\
&=& D(\mu, n, R, f).
\end{eqnarray*} In case of $n=0,$ first note that $f_{\!_R}\in\mathcal O(\overline{\mathbb D},\mathcal E)$  and the map $w\mapsto P_{\!\mu}(Rw)$ is continuous on the closed unit disc $\overline{\mathbb D}.$ Thus it follows from \eqref{eqnprop3.3 1} that
 \begin{align*}
\big\langle P_{\!\lambda_R}(r\zeta)f_{\!_R}(r\zeta),f_{\!_R}(r\zeta)\big\rangle \to \big\langle P_{\!\mu}(R\zeta)f(R\zeta),f(R\zeta)\big\rangle\quad \quad \text{as}\quad r\to1.
\end{align*}
Hence an application of dominated convergence theorem will give us 
\begin{eqnarray*}
D_{\lambda_R, 0}(f_{\!_R})
=\displaystyle \int_{\mathbb T} \big\langle P_{\!\mu}(R\zeta)f(R\zeta), f(R\zeta) \big\rangle \,d\sigma(\zeta)= D(\mu, 0, R, f).
\end{eqnarray*} 
This completes the proof of the first part of the proposition. For the last part, let $f\in \mathcal O(\mathbb D,\mathcal E)$ and $n\in\mathbb N.$ Note that for each $z\in\mathbb D,$ the function $\chi_{_{R\mathbb D}}(z)\big\langle P_{\!\mu} (z)f^{(n)}(z), f^{(n)}(z)\big\rangle  (R^2-|z|^2)^{n-1}$ increases to $\big\langle P_{\!\mu} (z)f^{(n)}(z), f^{(n)}(z)\big\rangle  (1^2-|z|^2)^{n-1}$ as $R$ increases to $1,$ where $\chi_A(z)$ denotes the characteristic function supported on $A.$ Hence by an application of monotone convergence theorem desired result follows.
\end{proof}

The following proposition establishes a difference identity which is a refinement of that in Lemma \ref{Rydhe diff vector valued}(ii) for $\mathcal E$-valued holomorphic functions on the open unit disc $\mathbb D.$ 
\begin{prop}\label{R diff vetor valued}
Let $\mu\in \mathcal M_+(\mathbb T,\mathcal B(\mathcal E)),$ $n\in\mathbb Z_+$  and  $0<R<1.$  Then for any $f\in \mathcal O(\mathbb D,\mathcal E),$ 
\begin{align*}
D(\mu, n+1, R, zf) - R^2D(\mu, n+1, R, f)=R^2 D(\mu, n, R, f).
\end{align*}
\end{prop}
\begin{proof}
 
Consider the $\mathcal B(\mathcal E)$-valued semi-spectral measure $\lambda _{R}$ on the unit circle $\mathbb T$ defined by $d\lambda_{R}(\zeta)= P_{\!\mu}(R\zeta)d\sigma(\zeta).$ 
By Proposition \ref{MCT of integral}, we have $D(\mu, n, R, f)= D_{\lambda_R, n}(f_{\!_R}).$ 
Note that $R^{n} (zf)^{(n+1)}(Rw)= (zf_{\!_R})^{(n+1)}(w)$ for every $w\in\mathbb D$.  
Again, by a change of variable, it is easy to verify that $D(\mu, n+1, R, zf) = R^2 D_{\lambda_R, n+1}\,(zf_{\!_R}).$
Hence, it follows that for every $f\in\mathcal O(\mathbb D,\mathcal E),$
\begin{align*}
D(\mu, n+1, R, zf) - R^2D(\mu, n+1, R, f)
=R^2 \big(D_{\lambda_R, n+1}(zf_{\!_R}) - D_{\lambda_R, n+1}(f_{\!_R})\big).
\end{align*}
Since $f_{\!_R}\in\mathcal O(\overline{\mathbb D},\mathcal E),$ an application of Lemma \ref{Rydhe diff vector valued}(ii) yields that
$$D(\mu, n+1, R, zf) - R^2D(\mu, n+1, R, f) = R^2 D_{\lambda_R, n}(f_{\!_R}).$$ 
This completes the proof.
\end{proof}

Let $\mu$ be a $\mathcal B(\mathcal E)$-valued semi-spectral measures on $\mathbb T$ and $n$ be an arbitrary but fixed positive integer. For the weighted Dirichlet integral $D_{\mu,n}(\cdot),$  we consider the associated weighted Dirichlet-type space $\mathcal H_{\mu,n}(\mathcal E)$  defined by
\begin{align*}
\mathcal H_{\mu,n}(\mathcal E):= \big\{f\in \mathcal O(\mathbb D,\mathcal E): D_{\mu,n}(f)<\infty\big\}.
\end{align*} 
It is straightforward to see that $\mathcal H_{\mu,n}(\mathcal E)$ is a linear subspace  of $\mathcal O(\mathbb D,\mathcal E),$ containing the set of all $\mathcal E$-valued polynomials.
Further, $\mathcal H_{\mu,n}(\mathcal E)$   is a semi-inner product space induced by the semi-norm $\sqrt{D_{\mu,n}(\cdot)}.$
If $\mathcal E= \mathbb C$  and $\mu$ is the Lebesgue measure $\sigma$ on $\mathbb T$ then the integral $D_{\sigma,1}(\cdot)$ becomes the usual Dirichlet integral and the space $\mathcal H_{\sigma,1}(\mathbb C)$ coincides with the classical Dirichlet space on $\mathbb D.$ 
More generally, the set of all monomials $\{z^j: j\geqslant 0\}$ forms an orthogonal set in $\mathcal H_{\sigma,n}(\mathbb C)$ and 
\begin{align*}
D_{\sigma,n}(z^j)=\begin{cases}
0 ,\quad 0\leqslant j\leqslant n-1,\\
\binom{j}{n}, \quad j\geqslant n.
\end{cases}
\end{align*}
The space $\mathcal H_{\mu,n}(\mathcal E)$ is a central topic of study in this section. In Proposition \ref{Bounded-H-mu-n vector valued}, we will show that the coordinate function $z$ is a multiplier of $\mathcal H_{\mu,n}(\mathcal E),$ that is, $f \in \mathcal H_{\mu, n}(\mathcal E)$ implies $zf\in \mathcal H_{\mu, n}(\mathcal E).$ 
Converse is also proved in Lemma \ref{Contractivity of L vector valued}(ii). To achieve this goal, we introduce another family of semi-norms which is a generalization of the class of norms introduced in \cite{TaylorDA} to the case of vector valued functions.

Suppose $Q$ is a positive operator in $\mathcal B(\mathcal E)$ and $\alpha\in \mathbb R$. Let $\mathcal D_{\alpha,Q}(\mathcal E):=\big\{f\in \mathcal O(\mathbb D, \mathcal E): \|f\|_{\mathcal D_{\alpha, Q}(\mathcal E)} < \infty\big\},$ where for any $f(z) =\sum_{k=0}^\infty a_kz^k$ in $\mathcal O(\mathbb D,\mathcal E)$, the semi-norm $\|f\|_{\mathcal D_{\alpha, Q}(\mathcal E)}$ is defined by
\begin{align}\label{Dalpha norm vector valued}
\|f\|^2_{\mathcal D_{\alpha, Q}(\mathcal E)}&= \sum\limits_{k=0}^{\infty}(k+1)^{\alpha}\langle Qa_k,a_k\rangle.
\end{align}  
It is easily verified that the space $\mathcal D_{\alpha, Q}(\mathcal E)$ equipped with the semi-norm $\|\cdot\|_{\mathcal D_{\alpha, Q}(\mathcal E)}$ is a semi-inner product space. Furthermore, if $Q$ is invertible then $\|\cdot\|_{\mathcal D_{\alpha, Q}(\mathcal E)}$ turns out to be a norm and the space $\mathcal D_{\alpha, Q}(\mathcal E)$ becomes a Hilbert space. 
Note that, if $Q$ is invertible then the two norms $\|\cdot\|_{\mathcal D_{\alpha, Q}(\mathcal E)}$ and $\|\cdot\|_{\mathcal D_{\alpha, I}(\mathcal E)}$ are equivalent and the associated spaces $\mathcal D_{\alpha, Q}(\mathcal E)$ and $\mathcal D_{\alpha, I}(\mathcal E)$ coincide. For any real number $\beta$ with $\alpha \leqslant \beta,$ it is straightforward to verify that $\|f\|_{\mathcal D_{\alpha, Q}(\mathcal E)}\leqslant \|f\|_{\mathcal D_{\beta, Q}(\mathcal E)}.$
Consequently, in this case, we have $\mathcal D_{\beta, Q}(\mathcal E) \subseteq \mathcal D_{\alpha, Q}(\mathcal E).$ 
When $\mathcal E=\mathbb C$ and $Q =1,$ the Hilbert space  $\mathcal D_{\alpha, Q}(\mathcal E)$ is simply denoted by $\mathcal D_{\alpha}.$ Many classical functional Hilbert spaces are given by $\mathcal D_{\alpha},$ for example, $\mathcal D_{-1},$ $\mathcal D_0$ and $\mathcal D_1$ are the Bergman space, the Hardy space and the Dirichlet space on the open unit disc $\mathbb D$ respectively.
From \eqref{Dalpha norm vector valued}, it can be easily seen that 
\begin{align}\label{derivative continuity vector-valued}
f\in \mathcal D_{\alpha, Q}(\mathcal E) \mbox{ if and only if } f' \in \mathcal D_{{\alpha-2}, Q}(\mathcal E).
\end{align}
For $\alpha <0$, the semi-norms $\|\cdot\|_{\mathcal D_{\alpha, Q}(\mathcal E)}$ and $\|\cdot\|_{\alpha,Q}$ on $\mathcal D_{\alpha, Q}(\mathcal E)$ are equivalent (a straightforward generalization of \cite[Lemma 2]{TaylorDA}), where $\|f\|_{\alpha, Q}$ is given by
\begin{align}\label{Integral form of Dalpha norm vector-valued}
\|f\|^2_{\alpha, Q}:= \int_{\mathbb D}\langle Qf(z), f(z)\rangle (1-|z|^2)^{-\alpha-1}dA(z).
\end{align}

For any $\mu \in \mathcal M_+(\mathbb T,\mathcal B(\mathcal E)),$ the following lemma establishes a relationship between the spaces $\mathcal H_{\mu,n}(\mathcal E)$ and $\mathcal D_{n,\mu(\mathbb T)}(\mathcal E).$ 
\begin{lemma} \label{embedding of spaces vector valued}
Let $\mu$ be a $\mathcal B(\mathcal E)$-valued semi-spectral measure on the unit circle $\mathbb T.$ Then, 
\begin{align*}
&\mathcal H_{\mu,n}(\mathcal E)\subseteq \mathcal D_{{n-1}, \mu(\mathbb T)}(\mathcal E), \quad n\geqslant 1,\\
&\mathcal D_{{n+1}, \mu(\mathbb T)}(\mathcal E) \subseteq \mathcal H_{\mu,n}(\mathcal E), \quad n\geqslant 2.
\end{align*}
\end{lemma}
\begin{proof}
From \cite[p. 236]{RudinRC}, it is easy to see that the Poisson kernel satisfies the following estimates:
\begin{align*}
\frac{1-|z|^2}{4}\leqslant  P(z,\zeta) \leqslant \frac{4}{(1-|z|^2)},\quad z\in\mathbb D,~ \zeta\in\mathbb T.
\end{align*}
This in turn implies that
\begin{align}\label{P-1-z vector-valued}
\frac{\mu(\mathbb T)}{4}(1-|z|^2)^n \leqslant P_{\!\mu}(z)(1-|z|^2)^{n-1}\leqslant 4\mu(\mathbb T)(1-|z|^2)^{n-2}, \quad n\in\mathbb N, \,z\in\mathbb D.
\end{align}
 Using \eqref{Integral form of Dalpha norm vector-valued} along with the above estimates in \eqref{P-1-z vector-valued}, it follows that for $f\in\mathcal O(\mathbb D, \mathcal E),$
\begin{align}
\label{embedding inequality vector-valued}
D_{\mu,n}(f) & \geqslant \frac{1}{4n! (n-1)!}\big\|f^{(n)}\big\|_{-(n+1),\mu(\mathbb T)}^2,\quad n\geqslant 1,\\
\label{embedding inequality1 vector-valued}
D_{\mu,n}(f) &\leqslant \frac{4}{n! (n-1)!} \big\|f^{(n)}\big\|_{-(n-1), \mu(\mathbb T)}^2,\quad n\geqslant 2.
\end{align}
Note that for $\alpha < 0,$ the semi-norm $\|f\|_{\alpha, \mu(\mathbb  T)}$ is equivalent to the semi-norm $\|f\|_{\mathcal D_{\alpha,  \mu(\mathbb  T)}(\mathcal E)}.$ Using this fact together with \eqref{derivative continuity vector-valued}, \eqref{embedding inequality vector-valued}, and \eqref{embedding inequality1 vector-valued}, we obtain the desired conclusion.
\end{proof}

\begin{cor}\label{containment in the Hardy space} 
Let  $n\in \mathbb N$ and $\mu\in \mathcal M_+(\mathbb T,\mathcal B(\mathcal E))$.  If $\mu(\mathbb T)$ is invertible then
$\mathcal H_{\mu,n}(\mathcal E)\subseteq H^2(\mathcal E).$ 
\end{cor}
\begin{proof}
Note that $\mathcal D_{{n-1}, \mu(\mathbb T)}(\mathcal E)\subseteq \mathcal D_{0, \mu(\mathbb T)}(\mathcal E)$ for every $n\in\mathbb N.$ Since $\mu(\mathbb T)$ is invertible, it follows that the two norms $\|\cdot\|_{\mathcal D_{0, \mu(\mathbb T)}(\mathcal E)}$ and $\|\cdot\|_{\mathcal D_{0, I}(\mathcal E)}$ are equivalent and the associated spaces $\mathcal D_{0, \mu(\mathbb T)}(\mathcal E)$ and $\mathcal D_{0, I}(\mathcal E)$ are equal. 
By definition, the space $\mathcal D_{0, I}(\mathcal E)$ is equal to the Hardy space $H^2(\mathcal E)$ of $\mathcal E$-valued holomorphic functions on the open unit disc $\mathbb D.$ The corollary is now immediate in the view of Lemma \ref{embedding of spaces vector valued}.
\end{proof}

Now we are ready to show that the coordinate function $z$ is a multiplier for $\mathcal H_{\mu,n}(\mathcal E).$ The proof is divided into two different cases.
Although the proof for the case of $n=1$ follows from an argument in \cite[Theorem 3.1]{OLLOF2iso}, nevertheless we provide an alternative proof for the sake of completeness.
\begin{prop}\label{Bounded-H-mu-n vector valued}
Let $\mu\in \mathcal M_+(\mathbb T,\mathcal B(\mathcal E))$ and $n\in\mathbb N.$ The coordinate function $z$ is a multiplier for the semi-inner product space $\mathcal H_{\mu,n}(\mathcal E).$
\end{prop}
\begin{proof} 
The proof is divided into two cases.

{\bf{Case $n=1$ :}} Let $f\in \mathcal H_{\mu, 1}(\mathcal E).$ Note that $f=f(0)+zg$ for some $g\in \mathcal O(\mathbb D, \mathcal E).$ As $P_{\!\mu}(z)$ is a positive operator, by an application of triangle inequality, it is straightforward to verify that
\begin{align*}
\big\langle P_{\!\mu}(z)f(z), f(z)\big\rangle \leqslant 2\big\langle P_{\!\mu}(z)f(0), f(0)\big\rangle +2 \big\langle P_{\!\mu}(z)zg(z), zg(z)\big\rangle,\,\,\,z\in\mathbb D.
\end{align*}
Let $R\in (0,1)$. Since the map $z\mapsto \langle P_{\!\mu}(Rz)f(0), f(0)\rangle$ is a positive harmonic function on a neighborhood of the closed unit disc $\overline{\mathbb D},$ by an application of the mean value property for harmonic functions, it follows that
\begin{align*}
\int_{\mathbb T}\big\langle P_{\!\mu}(R\zeta)f(0), f(0)\big\rangle\,d\sigma(\zeta) = \big\langle P_{\!\mu}(0)f(0), f(0)\big\rangle = \big\langle \mu(\mathbb T)f(0), f(0)\big\rangle.
\end{align*}
Thus we get that 
\begin{align*}
D(\mu,0,R,f)&= \displaystyle \int_{\mathbb T}\big\langle P_{\!\mu}(R\zeta)f(R\zeta), f(R\zeta)\big\rangle\, d\sigma(\zeta)\\
&\leqslant 2\big\langle \mu(\mathbb T)f(0), f(0)\big\rangle + 2 R^2 \displaystyle \int_{\mathbb T}\big\langle P_{\!\mu}(R\zeta)g(R\zeta), g(R\zeta)\big\rangle\, d\sigma(\zeta)\\
&= 2\big\langle \mu(\mathbb T)f(0), f(0)\big\rangle + 2 R^2D(\mu,0,R,g).
\end{align*}
Applying Proposition \ref{R diff vetor valued} to the function $g,$ we find that $R^2D(\mu,0,R,g) \leqslant D (\mu,1,R,zg).$ Thus we have 
\begin{align*}
D(\mu,0,R,f)\leqslant 2\big\langle \mu(\mathbb T)f(0), f(0)\big\rangle + 2 D(\mu,1,R,zg).
\end{align*} 
Note that $D(\mu,1,R,zg)= D(\mu,1,R,f).$ This together with Proposition \ref{R diff vetor valued} will give us
\begin{align}\label{Multiplier-case n=1}
D(\mu,1,R,zf)\leqslant 2\big\langle \mu(\mathbb T)f(0), f(0)\big\rangle + 3 D(\mu,1,R,f).
\end{align} 
Now taking limit $R\to1$ on both the sides of \eqref{Multiplier-case n=1}, we get $D_{\mu,1}(zf)\leqslant 2\big\langle \mu(\mathbb T)f(0), f(0)\big\rangle + 3 D_{\mu,1}(f).$ This shows that $zf\in \mathcal H_{\mu, 1}(\mathcal E)$ whenever $f\in \mathcal H_{\mu, 1}(\mathcal E).$

{\bf{Case $n\geqslant 2$ :}} 
Consider  $d\nu(z):=(1-|z|^2)^{n-1} P_{\!\mu}(z) dA(z),$ a $\mathcal B(\mathcal E)$-valued  weighted area measure on the unit disc $\mathbb D.$ 
Note that an $\mathcal E$-valued holomorphic function $f$ on the unit disc $\mathbb D$ is in $\mathcal H_{\mu, n}(\mathcal E)$ if and only if $f^{(n)}$ is in $L_a^2(\mathbb D, \mathcal E, d\nu),$ 
where 
\begin{equation*}
L_a^2(\mathbb D, \mathcal E, d\nu):=
\Big\{g\in \mathcal O(\mathbb D, \mathcal E): \int_{\mathbb D} \big\langle P_{\!\mu}(z)g(z), g(z)\big\rangle (1-|z|^2)^{n-1} dA(z)  < \infty\Big\}.
\end{equation*}
Let $f\in\mathcal H_{\mu, n}(\mathcal E).$ We easily see that $zf^{(n)}\in L_a^2(\mathbb D, \mathcal E, d\nu).$
In order to show that $zf\in \mathcal H_{\mu, n}(\mathcal E),$ using the relation $(zf)^{(n)}=zf^{(n)}+n f^{(n-1)},$ it suffices to prove that $f^{(n-1)}\in  L_a^2(\mathbb D, \mathcal E, d\nu).$
By Lemma \ref{embedding of spaces vector valued}, $\mathcal H_{\mu, n}(\mathcal E)\subseteq \mathcal D_{n-1, \mu(\mathbb T)}(\mathcal E)$ and hence $f\in \mathcal D_{n-1, \mu(\mathbb T)}(\mathcal E).$ By repeated use of \eqref{derivative continuity vector-valued}, one obtains $f^{(n-1)}\in \mathcal D_{-(n-1), \mu(\mathbb T)}(\mathcal E)$ and therefore $\|f^{(n-1)}\|_{-(n-1), \mu(\mathbb T)}< \infty.$
Using \eqref{Integral form of Dalpha norm vector-valued} and \eqref{P-1-z vector-valued}, we get
\begin{eqnarray*}
4\| f^{(n-1)}\|_{-(n-1), \mu(\mathbb T)}&= &4 \int_{\mathbb D}  \langle\mu(\mathbb T)f^{(n-1)}(z), f^{(n-1)}(z)\rangle (1-|z|^2)^{n-2}dA(z)\\
&\geqslant & \int_{\mathbb D} \langle P_{\!\mu}(z)
f^{(n-1)}(z), f^{(n-1)}(z)\rangle (1-|z|^2)^{n-1}dA(z).
\end{eqnarray*}
Hence  $f^{(n-1)}\in  L_a^2(\mathbb D, \mathcal E, d\nu),$ completing the proof of the proposition.
\end{proof}

As the coordinate function $z$ is a multiplier for $\mathcal H_{\mu,n}(\mathcal E),$ we obtain that the difference identity as described in Lemma \ref{Rydhe diff vector valued}(ii) remains valid for a larger class of functions, namely for functions in $\mathcal H_{\mu,n}(\mathcal E).$  
\begin{prop}\label{general diff. formula vector valued}
Let $\mu$ be a $\mathcal B(\mathcal E)$-valued semi-spectral measure on the unit circle $\mathbb T$ and $n$ be a positive integer. Then for every $f$ in $\mathcal H_{\mu,n}(\mathcal E),$ $D_{\mu,n}(zf)- D_{\mu,n}(f) = D_{\mu,n-1}(f).$ 
\end{prop}
\begin{proof}
Let $0 < R< 1$ and  $f\in \mathcal H_{\mu,n}(\mathcal E)$. By Proposition \ref{R diff vetor valued}, we have that 
\begin{align}\label{auxilary}
D(\mu, n, R, zf) - R^2D(\mu, n, R, f)=R^2 D(\mu, n-1, R, f).
\end{align} 
Note that, by Proposition \ref{Bounded-H-mu-n vector valued}, $zf\in \mathcal H_{\mu,n}(\mathcal E)$.
Thus by taking limit as $R\to 1$ on the both sides of \eqref{auxilary}, we obtain the desired identity.
\end{proof}
As an immediate corollary of the above proposition, we get the following interesting inclusion.
\begin{cor}\label{Containment of Hmu vector valued}
Let $\mu\in \mathcal M_+(\mathbb T,\mathcal B(\mathcal E))$ and $n\in\mathbb N$ and $f\in\mathcal O(\mathbb D,\mathcal E).$ If $D_{\mu,n}(f) < \infty,$  then    $D_{\mu,n-1}(f) < \infty.$ Consequently, we have $\mathcal H_{\mu,j+1}(\mathcal E) \subseteq \mathcal H_{\mu,j}(\mathcal E)$ for every $j\in\mathbb N.$
\end{cor}

In the following lemma, we present the converse of Proposition \ref{Bounded-H-mu-n vector valued}.
For any $f\in\mathcal O(\mathbb D, \mathcal E),$ let $Lf$ be the $\mathcal E$-valued function on the unit disc $\mathbb D,$ defined by 
\begin{align*}
Lf(z)&= \frac{f(z)-f(0)}{z},\quad z\in \mathbb D.
\end{align*}
Note that $Lf\in\mathcal O(\mathbb D, \mathcal E)$ and $D_{\mu,n}(zLf)=D_{\mu,n}(f)$ for every $n\in\mathbb N$ and $\mu\in \mathcal M_+(\mathbb T,\mathcal B(\mathcal E)).$

\begin{lemma}\label{Contractivity of L vector valued}
Let $\mu\in \mathcal M_+(\mathbb T,\mathcal B(\mathcal E))$ and $n\in\mathbb N.$
Then for every $f\in\mathcal O(\mathbb D, \mathcal E),$ we have  
\begin{itemize}
\item[(i)] $D_{\mu,n}(Lf)\leqslant D_{\mu,n}(f).$
\item[(ii)] $f\in\mathcal H_{\mu,n}(\mathcal E)$ if and only if $zf\in\mathcal H_{\mu,n}(\mathcal E).$
\end{itemize}
 \end{lemma}
\begin{proof}
Suppose $f\in\mathcal O(\mathbb D, \mathcal E).$ As an application of Proposition \ref{R diff vetor valued}, for any $0<R<1$ and $n\in\mathbb N,$ we obtain  
\begin{align*}
R^2D(\mu, n, R, f) \leqslant D(\mu, n, R, zf).
\end{align*}
Now taking limit as $R\to 1$ on both sides, we get that $D_{\mu,n}(f)\leqslant D_{\mu,n}(zf).$ Since  $D_{\mu,n}(zLf)=D_{\mu,n}(f),$ it follows that $D_{\mu,n}(Lf)\leqslant D_{\mu,n}(f)$ for every $f\in\mathcal O(\mathbb D,\mathcal E).$ This shows that $f\in\mathcal H_{\mu,n}(\mathcal E)$ whenever $zf\in\mathcal H_{\mu,n}(\mathcal E).$ The remaining part follows from Proposition \ref{Bounded-H-mu-n vector valued}.
\end{proof}

The following  lemma can be thought of as a generalization of \cite[Lemma 3.3]{R}. This will be an essential ingredient in proving Theorem \ref{Radial domination}. We will see in Corollary \ref{Improvement tool} of Section \ref{sec3}  that this lemma is also valid for $n\geqslant 1$ and $j=0$.
\begin{lemma}\label{backward difference identity vector valued}
Let $n\geqslant 2$ and $\mu$ be a $\mathcal B(\mathcal E)$-valued semi-spectral measure on $\mathbb T.$ Then for any function $f$ in $\mathcal H_{\mu,n}(\mathcal E)$, we have
\begin{eqnarray}\label{j greater 1}
\sum\limits_{k=1}^{\infty}D_{\mu,j}(L^kf)=D_{\mu,j+1}(f), \quad 1\leqslant j\leqslant n-1.
\end{eqnarray}
\end{lemma}
\begin{proof} Let $f\in \mathcal H_{\mu,n}(\mathcal E)$ and $k\in\mathbb N$. It follows from Lemma \ref{Contractivity of L vector valued} that $L^kf \in \mathcal H_{\mu,n}(\mathcal E)$ and hence by Corollary \ref{Containment of Hmu vector valued}, $D_{\mu, j}(L^kf)$ is also finite for all $j=1,\ldots, n-1$. 
%
Let $0<R<1$ and $j\in\{1,\ldots, n-1\}$. Since $D(\mu, j+1, R, f)
=D(\mu, j+1, R, zLf),$ by Proposition \ref{R diff vetor valued}, we obtain that
\begin{align*}
D(\mu, j+1, R, f)
= R^2D(\mu, j+1, R, Lf)+R^2 D(\mu, j, R, Lf).
\end{align*}
Applying this equality repeatedly, one obtains
\begin{equation}\label{eqn backward diff vector valued}
D(\mu, j+1, R, f)= R^{2k}D(\mu, j+1, R, L^kf)+\sum_{i=1}^k R^{2i}D(\mu, j, R, L^if).
\end{equation}
By Corollary \ref{Containment of Hmu vector valued}, we have $f\in \mathcal H_{\mu,j+1}(\mathcal E)$ and consequently from Lemma \ref{Contractivity of L vector valued} we have $L^k f\in \mathcal H_{\mu,j+1}(\mathcal E)$. Note that $D(\mu, j+1, R, L^kf)$ increases to $D_{\mu,j+1}(L^kf)$ as $R\to1$. Now by repeated applications of Lemma \ref{Contractivity of L vector valued},  we get $D(\mu, j+1, R, L^kf)\leqslant D_{\mu,j+1}(L^kf)\leqslant D_{\mu,j+1}(f).$
Thus, it follows that $\lim_{k\to \infty}R^{2k}D(\mu, j+1, R, L^kf)=0$. Hence, by \eqref{eqn backward diff vector valued}, the series $\sum_{i=1}^\infty R^{2i}D(\mu, j, R, L^if)$ is convergent and 
\begin{equation}\label{eqn backward diff 1 vector valued}
\sum_{i=1}^\infty R^{2i}D(\mu, j, R, L^if)=D(\mu, j+1, R, f).
\end{equation}
Since for all  $i\in \mathbb N$, $D(\mu, j, R, L^if)$ increases to $D_{\mu,j}(L^if)$ as $R\to 1$, an application of the monotone convergence theorem completes the proof.
\end{proof}

\section{Approximations through dilations}\label{sec3}
Let $\mu \in \mathcal M_+(\mathbb T,\mathcal B(\mathcal E))$ and $n$ be a positive integer. For $0<r<1$ and $f\in\mathcal O(\mathbb D, \mathcal E),$ let  $f_r$ denote the $r$-dilation of $f,$ that is,  $f_r(z) := f(rz),\,\,z\in\overline{\mathbb D}.$
In this section, we will show that for every  $f\in \mathcal H_{\mu,n}(\mathcal E),$ we have $f_r \to f$ in $\mathcal H_{\mu,n}(\mathcal E)$ as $r\to 1.$ A standard approach to obtain this is to find a positive constant $C$ such that $D_{\mu,n}( f_r)\leqslant C D_{\mu,n}(f)$ holds for every $0 < r < 1.$ 
In case of $n=1$ and $\mathcal E=\mathbb C,$ in \cite[Proposition 3]{Sara}, Sarason proved that $D_{\mu,1}(f_r) \leqslant  D_{\mu,1}(f)$ for every  $f\in \mathcal H_{\mu,1}(\mathbb C)$ and $0 < r< 1,$ (see also \cite[Theorem 5.2]{RS}, \cite[Lemma 4.1]{Aleman} and \cite[Lemma 7.3.2]{Primer}). 
%
In what follows, we show that $D_{\mu,n}( f_r)\leqslant  D_{\mu,n}(f)$, $0 < r < 1$, $f\in\mathcal H_{\mu,n}(\mathcal E)$ holds for any positive integer $n$ and any complex separable Hilbert space $\mathcal E.$ 
Before we provide a proof of this, we choose to draw a proof for the base case, that is, the case of $n=1$ and arbitrary $\mathcal E$ in the following lemma.

\begin{lemma}\label{contractive of fr}
Let $\mu$ be a $\mathcal B(\mathcal E)$-valued semi-spectral measure on $\mathbb T.$ Then for any $f\in \mathcal H_{\mu,1}(\mathcal E)$ and $0 < r< 1,$ the inequlity  $D_{\mu,1}(f_r) \leqslant  D_{\mu,1}(f)$ holds.  
\end{lemma}
\begin{proof}
Let $p$ be a $\mathcal E$-valued polynomial given by $p(z) = \sum_{j=0}^d c_j z^j.$ By Lemma \ref{Rydhe diff vector valued}, we have 
\begin{align}\label{poly form}
D_{\mu,1}(p)= \sum\limits_{k,l=1}^{d} (k\wedge l)\big\langle \hat{\mu}(l-k)c_k, c_l \big\rangle.
\end{align}
Consider the matrix  $A = (\!( A_{k,l})\!)_{k,l=0}^{\infty},$ where $A_{k,l} = (k\wedge l)\hat{\mu}(l-k)$ for $k,l \geqslant 0.$ In view of \eqref{poly form}, it follows that the  matrix $A$ is formally positive semi-definite. 
Let $\sigma^* A $ be the infinite matrix whose $(k,l)$-th element is given by $(\sigma^* A)_{k,l}= A_{k+1,l+1}$ for $k,l\geqslant 0.$ By  Proposition \ref{general diff. formula vector valued}, we also have 
$$D_{\mu,1}(z^2p)-2 D_{\mu,1}(zp) + D_{\mu,1}(p)= D_{\mu,0}(zp)-D_{\mu,0}(p) = 0.$$ 
This is equivalent to saying that $(\sigma ^* -I)^2 A=0.$ Thus from \cite[Theorem 3.11]{Shimorin}, it follows that the matrix $(\!( (1-r^{k+l})A_{k,l})\!)_{k,l=0}^{\infty}$ is formally positive semi-definite for every $0 < r < 1.$ 
This gives us $D_{\mu,1}(p_r)\leqslant D_{\mu,1}(p)$ for every $\mathcal E$-valued polynomial $p$ and $0 < r < 1.$ 
By a simple uniform limit argument, it follows that for each $r\in (0,1),$ we have  
\begin{align*}
D_{\mu,1}(f_r)\leqslant D_{\mu,1}(f),
\end{align*}
 whenever $f$ is a $\mathcal E$-valued holomorphic function defined on a neighbourhood of the closed unit disc, that is, $f\in\mathcal O(\overline{\mathbb D},\mathcal E).$ Now let $f\in \mathcal H_{\mu,1}(\mathcal E)$ and $r$ be an arbitrary but fixed number in $(0,1).$
Let $ 0< R < 1$ and $\lambda_R$ be the $\mathcal B(\mathcal E)$-valued semi-spectral measure on $\mathbb T$ given by $d\lambda_R(\zeta)= P_{\mu}(R\zeta)d\sigma (\zeta).$ Since $f_R$ is in $\mathcal O(\overline{\mathbb D},\mathcal E),$ we obtain $D_{\lambda_R,1}(({f_R})_r)\leqslant D_{\lambda_R,1}(f_R).$ 
Since $(f_R)_r= (f_r)_R,$ it follows that 
\begin{align*}
D_{\lambda_R,1}(({f_r})_R)\leqslant D_{\lambda_R,1}(f_R).
\end{align*}
Thus we obtain that $D(\mu, 1, R, f_r) \leqslant D(\mu, 1, R, f).$ This holds for every $R\in (0,1).$ By taking limit $R\to 1,$ we obtain the desired result. 
\end{proof}

Now we provide the proof of the approximation result, Theorem \ref{Radial domination}.
\begin{proof}[\textbf{Proof of Theorem \ref{Radial domination}}]
 Using induction, we shall first prove that $D_{\mu,n}(f_r)\leqslant D_{\mu,n}(f)$ for each $f\in\mathcal H_{\mu,n}(\mathcal E)$ and $0<r<1.$
Lemma \ref{contractive of fr} precisely deals with the case $n=1$. 
Fix a positive integer $n\geqslant 2$ and let the claim holds for every $f\in\mathcal H_{\mu,n}(\mathcal E)$ and $0<r<1.$ Let $f\in\mathcal H_{\mu,n+1}(\mathcal E)$ and $0 < r < 1.$ 
Note that $(Lf_r)(z)=r(Lf)_r(z) $ for every $z\in\mathbb D.$ A simple induction argument will give us $(L^kf_r)(z)=r^k(L^kf)_r(z) $ for every $k\in\mathbb N,$ $z\in\mathbb D.$ Thus we have $L^kf_r= r^k(L^kf)_r$ for every $k\in\mathbb N.$ It follows that $D_{\mu,n}(L^kf_r)= r^{2k} D_{\mu,n}((L^kf)_r)$ for every $k\in\mathbb N.$ As $L$ acts contractively on $\mathcal H_{\mu,n+1}(\mathcal E),$ we have $L^kf\in \mathcal H_{\mu,n+1}(\mathcal E)$ for every $k\in\mathbb N,$ see Lemma \ref{Contractivity of L vector valued}. Using Corollary \ref{Containment of Hmu vector valued}, we have $L^kf\in \mathcal H_{\mu,n}(\mathcal E)$ for every $k\in\mathbb N.$ Now applying induction hypothesis, we obtain that 
\begin{align*}
D_{\mu,n}(L^kf_r)= r^{2k} D_{\mu,n}((L^kf)_r) \leqslant r^{2k}D_{\mu,n}(L^kf) < D_{\mu,n}(L^kf),\,\,\,k\in\mathbb N.
\end{align*}
An application of Lemma \ref{backward difference identity vector valued} will give us 
\begin{align*}
D_{\mu,n+1}(f_r) = \sum\limits_{k=1}^{\infty}D_{\mu,n}(L^kf_r) \leqslant \sum\limits_{k=1}^{\infty}D_{\mu,n}(L^kf) =  D_{\mu,n+1}(f).
\end{align*}
This completes the proof for the first part of Theorem \ref{Radial domination}.

The technique to prove the remaining part of Theorem \ref{Radial domination} is standard, see for instance \cite[Theorem 7.3.1]{Primer}. Nevertheless we include  the details for the sake of completeness. For every positive integer $n,$ using parallelogram identity, we have 
\begin{align*}
D_{\mu,n}(f_r-f)+D_{\mu,n}(f_r+f)= 2 D_{\mu,n}(f_r) + 2D_{\mu,n}(f), \quad f\in \mathcal H_{\mu,n}(\mathcal E).
\end{align*}
Using the first part of Theorem \ref{Radial domination}, we have $D_{\mu,n}(f_r) \leqslant D_{\mu,n}(f)$ for every $0 < r <1.$ Note that for every $z\in\mathbb D,$ we have $f_r^{(n)}(z)\rightarrow f^{(n)}(z)$ as $r\rightarrow 1.$  Applying  Fatou's lemma, we obtain $D_{\mu,n}(2f) \leqslant \liminf\limits_{r\to 1} D_{\mu,n}(f_r+f).$ Hence it follows that 
\begin{align*}
\limsup\limits_{r\to 1} D_{\mu,n}(f_r-f) \leqslant 0.
\end{align*}
Since $D_{\mu,n}(f_r-f) \geqslant 0$ for every $0 < r <1,$ we conclude that $\lim\limits _{r \to 1}D_{\mu,n}(f_r-f)= 0.$ This completes the proof of the theorem. 
\end{proof}

\begin{cor}\label{polynomial approx}
Let $\mu$ be a $\mathcal B(\mathcal E)$-valued semi-spectral measure on the unit circle $\mathbb T,$ $n\in\mathbb N$ and $f\in \mathcal H_{\mu,n}(\mathcal E).$ Then there exists a sequence of polynomials $\{p_k\}$ such that $D_{\mu,n}(p_k-f)\to 0$ as $k\to\infty.$
\end{cor}
\begin{proof}
Let $\epsilon > 0.$ It is sufficient to show that there exists a polynomial $p$ such that $D_{\mu,n}(p-f) < \epsilon.$ By Theorem \ref{Radial domination}, there exists a $R\in (0,1)$ such that $D_{\mu,n}(f_{\!_R}-f) < \frac{\epsilon }{2}.$ Since  $f_{\!_R}\in \mathcal O(\overline{\mathbb D},\mathcal E),$ the power series expansion of $f_{\!_R}$ about origin converges uniformly on a neighbourhood of the closed unit disc $\overline{\mathbb D}.$ Let  $s_j(f_{\!_R})$ be the $j$-th partial sum of the associated power series of $f_{\!_R}.$ 
By Lemma \ref{Rydhe diff vector valued}(i), it follows that there exists $\ell\in\mathbb N$ such that $D_{\mu,n}(f_{\!_R} - s_{\ell}(f_{\!_R}))<\epsilon/2.$ Hence $D_{\mu,n}(f- s_{\ell}(f_{\!_R}))<\epsilon,$ completing the proof.
\end{proof}

We conclude this section with the following corollary which improves  Lemma \ref{backward difference identity vector valued} as promised earlier. The argument used  in the proof of the first part is essentially same as \cite[Remark, pg 210]{Rinv}.
\begin{cor}\label{Improvement tool}
Let $\mu$ be a $\mathcal B(\mathcal E)$-valued semi-spectral measure on the unit circle $\mathbb T,$ $n\in\mathbb N$ and $f\in \mathcal H_{\mu,n}(\mathcal E).$ Then the following statements hold:
\begin{itemize}
\item[\rm (i)]$D_{\mu,n}(L^kf) \to 0$ as $k\to\infty.$
\item[\rm (ii)]$\sum\limits_{k=1}^{\infty}D_{\mu,0}(L^kf)=D_{\mu,1}(f).$
\end{itemize}
\end{cor}
\begin{proof}
Let $\epsilon > 0.$ By Corollary \ref{polynomial approx}, there exists a polynomial $p$ such that $D_{\mu,n}(f-p)<\epsilon.$ Using Lemma \ref{Contractivity of L vector valued}, we obtain that
for all $k\geqslant {\rm deg}(p)+1$, $$D_{\mu,n}(L^kf)=D_{\mu,n}(L^k(f-p))\leqslant D_{\mu,n}(f-p)<\epsilon.$$
This completes the proof of the first part. For the second part, note that by Lemma \ref{Contractivity of L vector valued} and Corollary \ref{Containment of Hmu vector valued}, $D_{\mu,0}(L^kf)$ is finite for each $k\in\mathbb N$. Using Proposition 2.7, we get that for each $j\in\mathbb N,$
\begin{eqnarray*}
\sum_{k=1}^{j}D_{\mu,0}(L^kf) &=& \sum_{k=1}^{j}\big(D_{\mu,1}(zL^{k}f)-D_{\mu,1}(L^kf)\big)\\ 
&=& \sum_{k=1}^{j}\big(D_{\mu,1}(L^{k-1}f)-D_{\mu,1}(L^kf)\big)\\
&=& D_{\mu,1}(f) - D_{\mu,1}(L^{j}f).
\end{eqnarray*}
The proof is now completed by applying the first part.
\end{proof}

\section{Weighted Dirichlet-type Spaces and analytic $m$-isometric operators}

Let $\mu$ be a $\mathcal B(\mathcal E)$-valued semi-spectral measure on $\mathbb T,$ $j\in\mathbb N$ and $f$ be an arbitrary but fixed function in $\mathcal H_{\mu,j}(\mathcal E).$  For any $n\in \mathbb Z_+,$ we consider $\triangle^{n} D_{\mu,j}(f),$ the $n$-th order forward difference  of $D_{\mu,j}(f),$ defined by 
\begin{align*}
\triangle^{n} D_{\mu,j}(f) := \sum\limits_{k=0}^n (-1)^{n-k} {\binom{n}{k}} D_{\mu,j}(z^kf),\,\,\,f\in \mathcal H_{\mu,j}(\mathcal E).
\end{align*}
Note that $\triangle^{k+1} D_{\mu,j}(f)= \triangle^{k} D_{\mu,j}(zf)-\triangle^{k} D_{\mu,j}(f)$ holds for every $k\in\mathbb Z_+.$ Now an induction argument together with the application of Proposition \ref{general diff. formula vector valued}, it follows that $\triangle^n D_{\mu,j}(f) = D_{\mu,j-n}(f)$ for every $0\leqslant n \leqslant j.$ Since $f\in \mathcal H_{\mu,j}(\mathcal E),$ by Proposition \ref{Bounded-H-mu-n vector valued} and Corollary \ref{Containment of Hmu vector valued}, we have $D_{\mu, 0}(zf) < \infty$ and $D_{\mu, 0}(f) < \infty.$ It is also straightforward to verify that $D_{\mu, 0}(zf)-D_{\mu,0}(f)= 0.$
 Hence we obtain the following
\begin{align}\label{forward diff id vector valued}
\Delta^n D_{\mu, j}(f) &=
\begin{cases}
D_{\mu, j-n}(f),\quad 0\leqslant n \leqslant j,\\
0,\quad \quad \quad \quad \quad n\geqslant j+1,
\end{cases} \quad f\in\mathcal H_{\mu,j}(\mathcal E), \,j\in\mathbb N.
\end{align}

Let $m\geqslant 2$ and  $\pmb\mu=(\mu_1, \ldots ,\mu_{m-1})$ be an $(m-1)$-tuple of $\mathcal B(\mathcal E)$-valued semi-spectral measures on $\mathbb T.$ In this section we will introduce a Hilbert space $\mathcal H_{\pmb\mu}(\mathcal E),$ called weighted Dirichlet-type space associated to $(m-1)$-tuple $\pmb\mu$ of semi-spectral measures, on which the operator $M_z$ acts as an analytic $m$-isometry.  Let $\mathcal H_{\pmb\mu}(\mathcal E)$ denote the linear space given by
\begin{align*}
\mathcal H_{\pmb\mu} (\mathcal E):&= \bigcap\limits_{j=1}^{m-1} \mathcal H_{\mu_j,j}(\mathcal E)\bigcap H^2({\mathcal E})\\&=\Big\{f\in \mathcal O(\mathbb D,\mathcal E): D_{\mu_j,j}(f)<\infty ~\mbox{for}~j=1,\ldots,m-1\Big\}\bigcap H^2({\mathcal E}).
\end{align*}
We associate a norm $\|\cdot\|_{\pmb\mu}$ to the linear space $\mathcal H_{\pmb\mu}(\mathcal E)$ given by
\[\|f\|_{\pmb\mu}^2:=\|f\|^2_{\!_{H^2(\mathcal E)}}+\sum_{j=1}^{m-1} D_{\mu_j,j}(f),\]
where $\|f\|_{\!_{H^2(\mathcal E)}}$ denotes the Hardy norm of $f$ for any $f\in H^2(\mathcal E).$ 
Note that, if there exists a $j\in \{1,\ldots,m-1\}$ such that $\mu_j(\mathbb T)$ is invertible then by Corollary  \ref{containment in the Hardy space}, $\mathcal H_{\mu_j,j}(\mathcal E)\subseteq H^2(\mathcal E)$, and therefore
$\mathcal H_{\pmb\mu} (\mathcal E)$ coincides with $ \cap_{j=1}^{m-1} \mathcal H_{\mu_j,j}(\mathcal E).$
It is straightforward to verify that the linear space $\mathcal H_{\pmb\mu}(\mathcal E)$ is a Hilbert space with respect to the norm $\|\cdot\|_{\pmb\mu}.$ Let $z\in\mathbb D$ and $x\in\mathcal E.$ Consider the evaluation map $ev_{z,x}:\mathcal H_{\pmb\mu}(\mathcal E) \rightarrow \mathbb C$ defined by 
$ev_{z,x}(f)= \langle f(z),x\rangle,$  $f\in\mathcal H_{\pmb\mu}(\mathcal E).$ Since $\|f\|_{\!_{H^2(\mathcal E)}} \leqslant \|f\|_{\pmb\mu}$ for every $f\in \mathcal H_{\pmb\mu}(\mathcal E),$ it follows that the evaluation map $ev_{z,x}$ is bounded for every $z\in\mathbb D$ and $x\in\mathcal E.$ Thus the Hilbert space $\mathcal H_{\pmb\mu}(\mathcal E)$ is a reproducing kernel Hilbert space (see \cite{ARO,PAULRKHS} for definition and other basic properties of reproducing kernel Hilbert spaces). 
\begin{theorem}\label{bounded m iso vector valued}
Suppose $m\geqslant 2$ and $\pmb\mu=(\mu_1,\ldots ,\mu_{m-1})$ is an $(m-1)$-tuple of $\mathcal B(\mathcal E)$-valued semi-spectral measures on $\mathbb T.$
Then the multiplication operator $M_z$ on $\mathcal H_{\pmb\mu}(\mathcal E)$ is a bounded, analytic  $m$-isometry. 
\end{theorem}
\begin{proof}
By Corollary \ref{Contractivity of L vector valued}, $zf\in \mathcal H_{\pmb\mu}(\mathcal E)$ whenever $f\in\mathcal H_{\pmb\mu}(\mathcal E).$ Since $\mathcal H_{\pmb\mu}(\mathcal E)$ is a reproducing kernel Hilbert space, by closed graph theorem, it follows that $M_z$ on $\mathcal H_{\pmb\mu}(\mathcal E)$ is bounded. As $\mathcal H_{\pmb\mu}(\mathcal E)$ is contained in $ \mathcal O(\mathbb D,\mathcal E),$ the operator $M_z$ on $\mathcal H_{\pmb\mu}(\mathcal E)$ is analytic. Note that
\begin{align*}
\sum\limits_{j=0}^{m} \binom{m}{j}(-1)^{m-j}\|z^{j}f\|^2_{\pmb\mu}= \sum\limits_{k=1}^{m-1} \Delta^{m} D_{\mu_k,k}(f),\quad f\in\mathcal H_{\pmb\mu}(\mathcal E).
\end{align*}
In view of \eqref{forward diff id vector valued}, it follows that the operator $M_z$ on $\mathcal H_{\pmb\mu}(\mathcal E)$ is an $m$-isometry.
\end{proof}

Following Lemma \ref{Contractivity of L vector valued}(ii), we see that every function $f$ in $\mathcal H_{\pmb\mu}(\mathcal E)$ has the following decomposition.
\begin{align*}
f(z)= f(0) + zg(z), \,\quad g\in \mathcal H_{\pmb\mu}(\mathcal E).
\end{align*}
Since $\langle f(0) , zg\rangle_{\pmb\mu}= 0,$ for every $g\in \mathcal H_{\pmb\mu}(\mathcal E),$ it follows that 
\begin{align}\label{cokernel vector valued}
\ker M_z^*= \mathcal E.
\end{align}
It is straightforward to see that $\mathcal H_{\pmb\mu}(\mathcal E)$ contains the set of all $\mathcal E$-valued polynomials. 
In the following proposition, we show that the set of all $\mathcal E$-valued polynomials is dense in $\mathcal H_{\pmb\mu}(\mathcal E).$ 
In the case of $m=2,$ that is, when $\pmb \mu = \mu_1,$ this result follows from \cite[Theorem 1]{Rinv} together with \eqref{cokernel vector valued}, see also \cite[Corollary 3.1]{OLLOF2iso}. Here we obtain that this result remains true even when $\pmb \mu$ is an arbitrary $(m-1)$-tuple of semi-spectral measures with $m\geqslant 2.$ 

\begin{prop}\label{polydense through dilation}
Let $m\geqslant 2$ and $\pmb\mu=(\mu_1,\ldots ,\mu_{m-1})$ be an $(m-1)$-tuple of $\mathcal B(\mathcal E)$-valued semi-spectral measures on $\mathbb T.$ Let $f\in \mathcal H_{\pmb\mu}(\mathcal E).$ Then $f_r,$ the $r$-dilation of $f,$ converges to $f$ in $\mathcal H_{\pmb\mu}(\mathcal E)$  as $r\to 1$, that is, $\|f_r-f\|_{\pmb \mu}\to 0$ as $r\to 1.$ Consequently, the set of all $\mathcal E$-valued polynomials is dense in $\mathcal H_{\pmb\mu}(\mathcal E).$
\end{prop}
\begin{proof}
 Applying Theorem \ref{Radial domination}, we obtain that $D_{\mu_j,j}(f_r-f)\to 0$ as $r\to 1$ for each $j=1,\ldots,m-1.$ Since $f\in H^2(\mathcal E),$ it is straightforward to  verify that $\|f_r-f\|_{H^2(\mathcal E)}\to 0$ as $r\to 1.$ Hence it follows that $\|f_r-f\|_{\pmb \mu}\to 0$ as $r\to 1.$ So, for any given $\epsilon > 0,$ there exists a $R\in (0,1)$ such that $\|f_{\!_R}-f\|_{\pmb \mu} < \frac{\epsilon }{2}.$ Note that  $f_{\!_R}$ is in $\mathcal O(\overline{\mathbb D},\mathcal E).$ So the power series expansion of $f_{\!_R}$ about origin converges uniformly on a neighbourhood of the closed unit disc $\overline{\mathbb D}.$ Let  $s_n(f_{\!_R})$ be the $n$-th partial sum of the associated power series of $f_{\!_R}.$ Clearly $\|f_{\!_R} - s_n(f_{\!_R}) \|_{H^2(\mathcal E)} \to 0$ as $n\to\infty.$
By Lemma \ref{Rydhe diff vector valued}(i), we obtain that $D_{\mu_j,j}(f_{\!_R} - s_n(f_{\!_R}))\to 0$ as $n\to \infty$ for every $j=1,\ldots, m-1.$ 
This gives us that for any given $\epsilon > 0$ there exists $k\in\mathbb N$ such that $\|f_{\!_R}- s_n(f_{\!_R})\|_{\pmb \mu} < \frac{\epsilon}{2},$ for every $n\geqslant k.$ Consequently, we obtain that $\|f- s_k(f_{\!_R})\|_{\pmb \mu} < \epsilon.$
\end{proof}

\begin{rem}
Note that, if $\mathcal E= \mathbb C$ and $\pmb\mu =(\mu_1,\ldots,\mu_{m-1})$ is an $(m-1)$-tuple of finite positive Borel measures in $\mathcal M_+(\mathbb T)$ then by Theorem \ref{bounded m iso vector valued} and Corollary \ref{poly-dense vector valued}, it follows that the operator $M_z$ on $\mathcal H_{\pmb\mu}(\mathbb C)$ is a cyclic, analytic $m$-isometry. In this case, the space $\mathcal H_{\pmb\mu}(\mathbb C)$ coincides with $\mathcal D_{\overrightarrow{\mu}}^2$, where $\overrightarrow{\mu}= (\sigma,\mu_1,\ldots,\mu_{m-1})$ is an $m$-tuple of finite positive Borel measures in $\mathcal M_+(\mathbb T),$ as described in Rydhe's model for cyclic $m$-isometry, see \cite[p. 735]{Rydhe}.
\end{rem}

Using Lemma \ref{Contractivity of L vector valued}(i), we observe that $L$ is a bounded operator on $\mathcal H_{\pmb\mu}(\mathcal E).$ Thus the operator $L$ is a left inverse of $M_z$ with $\ker L= \ker M_z^*$ (see \eqref{cokernel vector valued}). Hence the operator $L$ on $\mathcal H_{\pmb\mu}(\mathcal E)$ coincides with the operator $L_{M_z}= (M_z^*M_z)^{-1}M_z^*.$ 
In what follows, we will use the notations $L$ and $L_{M_z}$ interchangeably, when the underlying Hilbert space is $\mathcal H_{\pmb\mu}(\mathcal E).$
The following theorem shows that the operator $M_z$ on  $\mathcal H_{\pmb\mu}(\mathcal E)$ satisfies \eqref{mu-positive}. In Section 6 we will show that this set of operator inequalities in \eqref{mu-positive} plays a key role in identifying  the operator $M_z$ on  $\mathcal H_{\pmb\mu}(\mathcal E)$ among the class of all analytic $m$-isometries.

\begin{theorem}\label{defect formula for model space vector valued}
Let $m\geqslant 2$ and $\pmb\mu=(\mu_1,\ldots ,\mu_{m-1})$ be an $(m-1)$-tuple of $\mathcal B(\mathcal E)$-valued semi-spectral measures on $\mathbb T.$ Then the operator $M_z$ on $\mathcal H_{\pmb \mu}(\mathcal E)$ satisfies
\begin{eqnarray}\label{Q-j and Q_j+1 vector valued}
\langle \beta_{r}(M_z) f, f\rangle = D_{\mu_r,0}(f) + \sum_{n=1}^{\infty}\big\langle \beta_{r+1}(M_z) L_{M_z}^n f, L_{M_z}^n f \big\rangle,  f\in\mathcal H_{\pmb \mu}(\mathcal E), r = 1,\ldots, m-1.
\end{eqnarray}
In particular, 
\begin{eqnarray*}
\beta_r(M_z)\geqslant \sum_{n=1}^{\infty} {L_{M_z}^*}^{\!\!\!n}\beta_{r+1}({M_z})L_{M_z}^n, \quad r = 1,\ldots, m-1.
\end{eqnarray*}
\end{theorem}
\begin{proof}

Since $\|zf\|_{H^2(\mathcal E)}=\|f\|_{H^2(\mathcal E)}$ for every $f\in\mathcal H_{\pmb\mu}(\mathcal E),$ it follows that for $r=1,\ldots,m-1,$
\begin{eqnarray*}
\langle \beta_r(M_z) f , f \rangle = \sum_{j=1}^{m-1} \Delta^r D_{\mu_j,j}(f),\quad f\in\mathcal H_{\pmb\mu}(\mathcal E).
\end{eqnarray*}
Note that if $\mu\in \mathcal M_+(\mathbb T,\mathcal B(\mathcal E))$ and $f\in \mathcal H_{\mu,j}(\mathcal E)$ for some $j\in\mathbb N,$ then in view of \eqref{forward diff id vector valued}, we get that 
$$
\triangle ^{s+1}D_{\mu,j}( L^nf) = \begin{cases}
D_{\mu,j-s-1}(L^nf), & ~\mbox {if}~~j-s-1\geqslant 0\\
0, & ~\mbox{otherwise}
\end{cases},\,\,n\in\mathbb N,\,\,s\in\mathbb Z_+.
$$ 
Consecutively using Lemma \ref{backward difference identity vector valued} and Corollary \ref{Improvement tool}, we get the  following identity:
\begin{eqnarray}\label{Important id}
\triangle ^{s}D_{\mu,j}(f) -\sum_{n=1}^{\infty}\triangle ^{s+1}D_{\mu,j}( L^nf)=
\begin{cases}
D_{\mu,0}(f), & s=j\\
0, & s\neq j
\end{cases},\,\,f\in\mathcal H_{\mu,j}(\mathcal E),\,j\in \mathbb N,\, s\in \mathbb Z_+. 
\end{eqnarray}
Using \eqref{Important id}, we obtain that for any $ f\in\mathcal H_{\pmb\mu}(\mathcal E),$
\begin{eqnarray*}
\langle \beta_r({M_z}) f , f \rangle - \sum_{n=1}^{\infty}\big\langle \beta_{r+1}({M_z}) L^n f, L^n f \big\rangle
=&\hspace*{-.2cm} \displaystyle\sum_{j=1}^{m-1} \Delta^r D_{\mu_j,j}(f) - \displaystyle\sum_{j=1}^{m-1}\sum_{n=1}^{\infty} \Delta^{r+1} D_{\mu_j,j}(L^nf)\\
=&\hspace*{-5.7cm} D_{\mu_r,0}(f).
\end{eqnarray*}
This establishes \eqref{Q-j and Q_j+1 vector valued} and completes the proof of the theorem.
\end{proof}


\section{The wandering subspace property for a class of $m$-concave operators}\label{sec5}
The main aim of this section is to establish Theorem \ref{Wold-type-mconcave}, that is, the wandering subspace property for left invertible analytic $m$-concave operators satisfying the operator inequalities \eqref{mu-positive}. This not only becomes a key tool to prove the main theorem of Section 6 but it also provides an alternative way to obtain the density of polynomials in the Hilbert space $\mathcal H_{\pmb\mu}(\mathcal E)$. Before proceeding to the proof, we note down the following lemma which will be crucial in the proof of Theorem \ref{Wold-type-mconcave}. 
This lemma can essentially be found in \cite[Corollary $2.4$ and Theorem $2.5$]{CG}.
\begin{lemma}\label{m-concaveop}
Let $T$ be an $m$-concave operator in $\mathcal B(\mathcal H)$ for some $m\in\mathbb N$. Then, 
\begin{eqnarray*}
{T^*}^nT^n\leqslant \sum_{j=0}^{m-1}\binom{n}{j}\beta_{j}(T),  \quad n\geqslant m.
\end{eqnarray*}
Moreover, $\beta_{m-1}(T)\geqslant 0.$
\end{lemma}

An operator $T$ in $\mathcal B(\mathcal H)$  is left invertible if and only if $T^*T$ is invertible. 
Note that $L_T$ is a left inverse of $T$ satisfying $\ker L_T=\ker T^*.$ Furthermore, if $L$ is a left inverse of $T$ satisfying $\ker L=\ker T^*$ then it follows that $L=L_T.$ The following lemma for a class of left invertible operators will be essential in proving Theorem \ref{Wold-type-mconcave}.
\begin{lemma}\label{Dm}
Let $m\geqslant 2$ and let $T$ be a left invertible operator in $\mathcal B(\mathcal H)$ satisfying 
\begin{itemize}
\item[\rm (i)]$\beta_{m-1}(T)\geqslant 0,$
\item[\rm (ii)]$\beta_{r}(T)\geqslant \displaystyle\sum_{n=1}^\infty {L_T^*}^n\beta_{r+1}(T)L_T^n, \quad  r=1,\ldots,m-2.$ 
\end{itemize}
Then the following inequalities hold
\begin{equation*}
\sum_{n=r+1}^\infty {\binom{n-1}{r}}{L_T^*}^n\beta_{r+1}(T)L_T^n \leqslant I, \quad r=0,\ldots,m-2.
\end{equation*}
\end{lemma}
\begin{proof}
It is evident from the hypothesis that $\beta_r(T)\geqslant 0$ for $r=1,\ldots, m-1.$ Also, for $r=0,\ldots,m-2$, set 
$\Psi(r)=\sum_{n=r+1}^\infty {\binom{n-1}{r}}{L_T^*}^n\beta_{r+1}(T)L_T^n$.
We claim that $\Psi(r)\leqslant I$ for $r=0,\ldots,m-2.$ To this end,
note that, for any $r\in \{1,\ldots, m-2\},$ we have
\begin{align*}\Psi(r-1)=
\sum_{i=r}^\infty \binom{i-1}{r-1} {L_T^*}^i\beta_{r}(T)L_T^i & \geqslant \sum_{i=r}^\infty \binom{i-1}{r-1}\sum_{n=1}^\infty {L_T^*}^{n+i}\beta_{r+1}(T)L_T^{n+i}\\
&=\sum_{p=r+1}^\infty \Bigg(\sum_{i=r}^{p-1} \binom{i-1}{r-1}\Bigg){L_T^*}^p\beta_{r+1}(T)L_T^p\\&
=\sum_{p=r+1}^\infty \binom{p-1}{r}{L_T^*}^p\beta_{r+1}(T)L_T^p=\Psi(r),
\end{align*}
where the second last equality follows from a combinatorial identity, known as the Hockey-stick identity (see \cite[p. 46]{DW}). Thus the above inequality shows that 
\begin{equation*}
\Psi(m-2)\leqslant \Psi(m-3)\leqslant\cdots\leqslant\Psi(0).
\end{equation*}
Hence, in order to prove $\Psi(r)\leqslant I$ for $r=0,\ldots,m-2$, it suffices to show that 
\begin{eqnarray*}
\Psi(0)=
\sum_{i=1}^\infty  {L_T^*}^i\beta_{1}(T)L_T^i\leqslant I.
\end{eqnarray*}
To this end, it follows from \cite[p. 209]{Rinv} that,
\begin{eqnarray*}
\|x\|^2= \sum_{i=0}^{n-1}\|PL_T^i x\|^2 + \|L_T^n x\|^2 + \sum_{i=1}^{n}\|DL_T^i x\|^2,
\end{eqnarray*}
where $D$ is the positive square root of $\beta_1(T)$ and $P=I-TL_T$. 
Thus we have, 
\[\Big\langle \sum_{i=1}^{\infty}{L_T^*}^i\beta_1(T) L_T^i x, x \Big\rangle = \sum_{i=1}^{\infty}\|DL_T^i x\|^2 \leqslant \|x\|^2.\]
Consequently, $\displaystyle\sum_{i=1}^\infty  {L_T^*}^i\beta_{1}(T)L_T^i\leqslant I$. 
This completes the proof.
\end{proof}
We are now ready to present the proof of the main  result (Theorem \ref{Wold-type-mconcave}) of this section. 
The techniques involved in the proof are motivated from those in \cite[Theorem 1]{R}.
\begin{proof}[\textbf{Proof of Theorem \ref{Wold-type-mconcave}}]
Since $T$ is $m$-concave, by Lemma \ref{m-concaveop}, $\beta_{m-1}(T)\geqslant 0$. This together with \eqref{mu-positive} implies that $\beta_1(T)\geqslant 0,$ i.e. $T$ is expansive. We claim that 
$$\bigvee\big\{T^n (\ker T^*):n\in\mathbb Z_+\big\}=\mathcal H.$$ 
Using Lemmas \ref{m-concaveop} and \ref{Dm}, note that for $x\in\mathcal H$ and for $k,l\in\mathbb N$ with $l\geqslant k \geqslant m,$ we have
\begin{eqnarray*}
\inf_{k\leqslant n \leqslant l}\big(\|T^nL_T^n x\|^2- \|L_T^nx\|^2\big) \sum_{n=k}^{l}\frac{1}{n}&\leqslant &\sum_{n=k}^{l} \frac{\|T^nL_T^n x\|^2- \|L_T^nx\|^2}{n}\\
&\leqslant & \sum_{n=k}^{l}\sum_{j=1}^{m-1} \frac{1}{j}\binom{n-1}{j-1}\langle {L_T^*}^n\beta_j(T) L_T^n x, x \rangle\\
& = & \sum_{j=1}^{m-1} \frac{1}{j}\sum_{n=k}^{l}{\binom{n-1}{j-1}}\langle {L_T^*}^n\beta_j(T) L_T^n x, x \rangle\\
& \leqslant & \Big(\sum_{j=1}^{m-1} \frac{1}{j}\Big) \|x\|^2.
\end{eqnarray*} 
Since $\{\|L_T^nx\|\}$ is a decreasing sequence of non-negative numbers and $\sum_{n=1}^{\infty}\frac{1}{n}$ is a divergent series, we have $\liminf \|T^n L_T^n x\|=\lim \|L_T^n x\|$. 
Thus the sequence $\{T^n L_T^n x\}$ is bounded in $\mathcal H$ and therefore there exists a subsequence $\{T^{n_k} L_T^{n_k} x\}$, converging to $y$ (say) in weak topology of $\mathcal H$. 
As $T^n$ is expansive, the range ${\rm ran}(T^n)$ is a closed subspace of $\mathcal H$ and $y\in {\rm ran}(T^n)$ for each $n\in\mathbb N.$ 
Since $T$ is analytic, $y=0.$ Thus $(I-T^{n_k}L_T^{n_k})x\to x$ weakly. Note that 
\begin{align*}
(I-T^jL_T^j)= \sum\limits_{p=0}^{j-1}T^p(I-TL_T)L_T^p,\,\,\,j\geqslant 1,
\end{align*}
and $(I-TL_T)$ is the orthogonal projection onto $\ker T^*.$ Thus it follows that $(I-T^{n_k}L_T^{n_k})x \in \bigvee\big\{T^n (\ker T^*):n\in\mathbb Z_+\big\}$. Hence $x\in \bigvee\big\{T^n (\ker T^*):n\in\mathbb Z_+\big\}.$
Thus we conclude that $T$ has the wandering subspace property.
\end{proof}

\begin{cor}\label{poly-dense vector valued}
Let $m\geqslant 2$ and $\pmb\mu=(\mu_1,\ldots ,\mu_{m-1})$ be an $(m-1)$-tuple of $\mathcal B(\mathcal E)$-valued semi-spectral measures on $\mathbb T.$ The multiplication operator $M_z$ on $\mathcal H_{\pmb\mu}(\mathcal E)$ has the wandering subspace property. In particular, the set of $\mathcal E$-valued polynomials is dense in $\mathcal H_{\pmb\mu}(\mathcal E).$
\end{cor}
\begin{proof}
By Theorem \ref{bounded m iso vector valued}, $M_z$ on $\mathcal H_{\pmb\mu}(\mathcal E)$ is a bounded, analytic $m$-isometry. Using Theorem \ref{defect formula for model space vector valued} and Theorem \ref{Wold-type-mconcave}, we get that $M_z$ has wandering subspace property. Since $\ker M_z^*=\mathcal E$ (see \eqref{cokernel vector valued}), the corollary is proved.
\end{proof}

\begin{rem}
We note here that the operator inequalities in \eqref{mu-positive} are not necessary for an expansive analytic $m$-isometry to have the wandering subspace property. For a counter-example (this example is due to Shailesh Trivedi) consider $m=3.$ 
Let $0<\varepsilon<1$ and $p(z)=1+z+\frac{\varepsilon}{2} z^2$ for $z\in\mathbb C.$ Define $\lambda_0:=\sqrt{1-\varepsilon},$ $\lambda_1:=\sqrt{2\varepsilon}$ and $\lambda_n:=\sqrt{\frac{p(n-1)}{p(n-2)}}$ for $n\geqslant 2$. 
Let $\{e_n: n\in\mathbb Z_+\}$ be an orthonormal basis for $\ell^2(\mathbb Z_+)$.
Define a linear operator 
$$S_\lambda : \ell^2(\mathbb Z_+)\to \ell^2(\mathbb Z_+)$$ 
by the rule
\begin{center}
$S_\lambda(e_n)=
\begin{cases}
\lambda_0 e_0 + \lambda_1 e_1 & \mbox{ if }n=0,\\
\lambda_{n+1}e_{n+1} & \mbox{ if } n\geqslant 1.
\end{cases}
$
\end{center}
From \cite[Example 3.1]{AST} (with $a=1$ and $b=\varepsilon/2$), we note that $S_\lambda$ is an analytic expansive $3$-isometry having wandering subspace property. 
We claim that if $\varepsilon$ is in a small neighbourhood of $0$ then $S_\lambda$ does not satisfy \eqref{mu-positive}.
From  \cite[Example 3.1]{AST}, we also note that $L^*_{S_\lambda}=S_\mu,$ where $S_\mu:\ell^2(\mathbb Z_+)\to \ell^2(\mathbb Z_+)$ is a linear operator such that 
\begin{center}
$S_\mu(e_n)=
\begin{cases}
\mu_0 e_0 + \mu_1 e_1 & \mbox{ if }n=0,\\
\mu_{n+1}e_{n+1} & \mbox{ if } n\geqslant 1.
\end{cases}
$
\end{center}
with $\mu_0=\frac{\lambda_0}{\lambda_0^2+\lambda_1^2},$ $\mu_1=\frac{\lambda_1}{\lambda_0^2+\lambda_1^2}$ and $\mu_n=\frac{1}{\lambda_n}$ for $n\geqslant 2.$ It follows that 
\begin{eqnarray}\label{Cauchy Dual}
L_{S_\lambda}^ke_0=\mu_0^k e_0, \quad k\in\mathbb N.
\end{eqnarray}
We have
\begin{eqnarray}\label{LHS}
\langle \beta_1(S_\lambda)e_0, e_0\rangle = \|S_\lambda e_0\|^2 - \|e_0\|^2 = \lambda_0^2+\lambda_1^2-1=\varepsilon
\end{eqnarray}
and from \eqref{Cauchy Dual}, one gets
\begin{eqnarray}\label{RHS}
\sum_{k=1}^{\infty} \Big\langle {{L_{S_\lambda}^*}^{\!\!\!\!k}}\beta_2(S_\lambda) L_{S_\lambda}^k e_0, e_0 \Big\rangle &=& \sum_{k=1}^{\infty} \mu_0^{2k} \langle \beta_2(S_\lambda)e_0, e_0\rangle\notag \\
&=& \frac{\mu_0^2}{1-\mu_0^2}(1-2(\lambda_0^2+\lambda_1^2)+\lambda_0^4+\lambda_1^2(\lambda_0^2+\lambda_2^2)) \notag\\
&=& \frac{2(1-\varepsilon)}{3+\varepsilon}
\end{eqnarray}
\end{rem}
If $\varepsilon$ is in a small neighbourhood of $0$  then we note that the quantity in \eqref{LHS} is lesser than that in \eqref{RHS} and hence the claim stands verified. 

We find that any left invertible $m$-concave operator in $\mathcal B(\mathcal H)$ satisfying \eqref{mu-positive} admits a \emph{Wold-type decomposition}, see Theorem \ref{mconcave new}. Shimorin  introduced the notion of \emph{Wold-type decomposition} in order to study operators close to isometries, see \cite{Shimorin}. An operator $T$ in $\mathcal B(\mathcal H)$ is said to admit a {\it Wold-type decomposition} if the following statements hold: 
\begin{itemize}
\item[(i)] $\mathcal H_{\infty}(T),$ the hyper range of $T,$ is a reducing subspace for $T$ and $T|_{\mathcal H_{\infty}(T)}$ is a unitary operator.
\item[(ii)] The operator $T|_{{\mathcal H_{\infty}(T)}^{\perp}}$  has the wandering subspace property.
\end{itemize}
It follows from \cite[Theorem 1]{Rinv} that every $2$-concave operator admits a Wold-type decomposition, see \cite[Theorem 3.6]{Shimorin}. In the same paper \cite{Shimorin}, Shimorin asked the following question:
\begin{Question}{\cite[p. 185]{Shimorin}}
If an operator $T$ in $\mathcal B(\mathcal H)$ is expansive and $m$-concave for some $m\geqslant 3$, then does $T$ admit a Wold-type decomposition?
\end{Question}
 The answer is not yet known even for the class of expansive $m$-isometries. Recently, it has been shown that there are plenty of non-expansive cyclic analytic $3$-isometries which fail to have the wandering subspace property (see \cite{AST}). In view of the above question, 
in case of $m$-concave operators, the best known result till now to our knowledge is the following theorem due to Shimorin (see \cite[Theorem 3.8]{Shimorin}).
\begin{theorem}[Shimorin]\label{Shimorin 3 concave}
Let $T\in \mathcal B(\mathcal H) $ be expansive and satisfy the operator inequality
\begin{eqnarray*}
{T^*}^2 T^2 - 3 T^*T + 3I - {T^{\prime}}^*T^{\prime} - P_{\ker T^*} \leqslant 0,
\end{eqnarray*}
where $P_{\ker T^*} $ is the orthogonal projection of $\mathcal H$ onto $\ker T^*$ and $T^{\prime}= T(T^*T)^{-1}$ is the Cauchy dual of $T.$ Then $T$ is a $3$-concave operator and admits a Wold-type decomposition.
\end{theorem}
In the following theorem we provide an improvement of Theorem \ref{Shimorin 3 concave}. Also note that the inequalities stated in this theorem are nothing but the inequalities in \eqref{mu-positive}.
\begin{theorem}\label{mconcave new}
Let $T$ be a left invertible $m$-concave operator in $\mathcal B(\mathcal H)$ for some $m \geqslant 2.$ If $T$ satisfies the following inequalities:
\begin{eqnarray*}
\beta_r(T)\geqslant \sum_{k=1}^{\infty} {L_T^*}^k\beta_{r+1}(T) L_T^k, \quad r=1,\ldots, m-2 
\end{eqnarray*}
then $T$ admits a Wold-type decomposition.
\end{theorem}
\begin{proof}
Since $T$ is $m$-concave, by Lemma \ref{m-concaveop}, $\beta_{m-1}(T)\geqslant 0$. This together with \eqref{mu-positive} implies that $\beta_1(T)\geqslant 0,$ i.e. $T$ is expansive. Hence,
by \cite[Proposition 3.4]{Shimorin}, $\mathcal H_\infty(T)$ is a reducing subspace for $T$ and $T|_{\mathcal H_\infty(T)}$ is a unitary operator. Define $S:=T|_{\mathcal H_\infty(T)^\perp}$ and note that $L_S=L_T|_{\mathcal H_\infty(T)^\perp}.$ 
It is straightforward to see that $S$ is analytic, $m$-concave and satisfies \eqref{mu-positive}. The wandering subspace property of $S$ now follows from Theorem \ref{Wold-type-mconcave}. This completes the proof.
\end{proof}
We also find that the result of Shimorin \cite[Theorem 3.8]{Shimorin} follows as a special case of Theorem \ref{mconcave new}. We provide the details below.
\begin{proof}[\textbf{A proof of Theorem \ref{Shimorin 3 concave}}]
Let $T$ be an expansive operator in $\mathcal B(\mathcal H)$ satisfying the inequality 
\begin{eqnarray}\label{Shimorin}
{T^*}^2T^2-3T^*T+3I-L_T^*L_T- P \leqslant 0,
\end{eqnarray}
where $P$ is the orthogonal projection onto $\ker T^*.$
Since $TL_T$ is an orthogonal projection onto the range of $T,$  we have $I-P=TL_T= L_T^*T^*TL_T. $ Note that \eqref{Shimorin} 
is equivalent to 
\begin{eqnarray}\label{Shimorin-transformed}
\beta_2(T) \leqslant \beta_1(T) - L_T^*\beta_1(T) L_T.
\end{eqnarray}
From \eqref{Shimorin-transformed}, we see that $T^*\beta_2(T)T\leqslant T^*\beta_1(T)T - \beta_1(T)=\beta_2(T),$ which in turn implies that $\beta_3(T)\leqslant 0.$ Hence $T$ is $3$-concave. Again, from \eqref{Shimorin-transformed},  we get 
\begin{eqnarray}\label{telescoping}
\sum_{k=0}^{n}{L_T^*}^k\beta_2(T) L_T^k \leqslant  \beta_1(T) - {L_T^*}^{n+1}\beta_1(T) L_T^{n+1}\leqslant  \beta_1(T), \quad n\in\mathbb N.
\end{eqnarray}
Since $T$ is $3$-concave, $\beta_2(T)\geqslant 0.$ Hence, \eqref{telescoping} in particular implies that
\[\sum_{k=1}^{\infty}{L_T^*}^k\beta_2(T) L_T^k \leqslant  \beta_1(T).\]
Now applying Theorem \ref{mconcave new} for $m=3$, we see that $T$ admits a Wold-type decomposition.
\end{proof}


\begin{rem}\label{remark strct containment vector valued} Note that by Theorem \ref{bounded m iso vector valued} and Theorem \ref{defect formula for model space vector valued},  for any $\pmb\mu=(\mu_1,\ldots ,\mu_{m-1}),$ an $(m-1)$-tuple of finite positive Borel measures on $\mathbb T,$ the multiplication operator $M_z$ on $\mathcal H_{\pmb \mu}(\mathbb C)$ is an $m$-isometry and satisfies \eqref{mu-positive}.  Now let $m=3$ and suppose that $M_z$ on $\mathcal H_{\pmb \mu}(\mathbb C)$ satisfies \eqref{Shimorin}. Then using the equivalence of \eqref{Shimorin} and \eqref{Shimorin-transformed} we have
\begin{equation*}
\big\langle\beta_2(M_z)1,1\big\rangle\leqslant\big\langle\big(\beta_1(M_z)-L_{M_z}^*\beta_1(M_z)L_{M_z}\big)1,1\big\rangle=\big\langle\beta_1(M_z)1,1\big\rangle.
\end{equation*}
Thus, by \eqref{Q-j and Q_j+1 vector valued}, we obtain $D_{\mu_2,0}(1)\leqslant D_{\mu_1,0}(1)$, that is, $\mu_2(\mathbb T)\leqslant \mu_1(\mathbb T).$
Hence we infer that the class of left invertible $3$-concave operators which satisfy \eqref{mu-positive} is strictly larger than that of expansive operators which satisfy \eqref{Shimorin}.
\end{rem}

\begin{rem}
 
Let $T$ be a left invertible $m$-concave operator in $\mathcal B(\mathcal H)$ with $m\geqslant 2.$
Note that for $e\in \ker T^*,$  we have
\begin{align}\label{mu-positive-diagonal}
\Big\langle\Big(\beta_r(T)-\sum_{k=1}^{\infty} {L_T^*}^k \beta_{r+1}(T) L_T^k\Big) T^n e, T^n e\Big\rangle &=\Big\langle\big({T^*}^n\beta_r(T)T^n-\sum_{k=1}^{n}{T^*}^{n-k} \beta_{r+1}(T)T^{n-k}\big)e, e\Big\rangle\nonumber\\
&=\Big\langle\big({T^*}^n\beta_r(T)T^n-\sum_{k=0}^{n-1}{T^*}^k \beta_{r+1}(T)T^k\big)e, e\Big\rangle\nonumber\\
&= \big\langle \beta_r(T) e,e\big\rangle, \quad r\in\mathbb N,\,\,n\in\mathbb Z_+.
\end{align}
Here the last equality follows from the relation
${T^*}^n\beta_r(T)T^n-\sum_{k=0}^{n-1}{T^*}^k \beta_{r+1}(T)T^k=\beta_r(T),$ which can easily be verified  by induction on $n$.
Further, suppose that $T$ is a unilateral weighted shift operator with non-zero weights 
(see \cite{SH} for definition and other basic properties of unilateral weighted shift). 
Then for any fixed non-zero vector $e\in\ker T^*,$ the set $\{T^ne:n\in\mathbb Z_+\}$ forms an orthogonal basis of $\mathcal H$. 
%
Note that with respect to this basis, the operator ${L_T^*}^k {T^*}^jT^j L_T^k$ is diagonal for each $j,\ k\in\mathbb Z_+$ and consequently the operator $ {L_T^*}^k \beta_{n}(T) L_T^k$ is also diagonal for each $n,k \in\mathbb Z_+.$ Hence in view of \eqref{mu-positive-diagonal}, it follows that if $\beta_r(T) \geqslant 0$ for $r=1,\ldots,m-2,$ then $\beta_r(T)\geqslant  \sum_{k=1}^{\infty} {L_T^*}^k \beta_{r+1}(T) L_T^k$ holds for $r=1,\ldots,m-2.$
Thus by an application of Lemma \ref{m-concaveop}, we obtain that for any left invertible $m$-concave unilateral weighted shift $T$ with $m\geqslant 2,$ the following two conditions are equivalent:
\begin{itemize}
\item[(i)] $\beta_r(T)\geqslant \sum_{k=1}^{\infty} {L_T^*}^k \beta_{r+1}(T) L_T^k$ for every $r=1,\ldots,m-2,$
\item[(ii)] $\beta_r(T)\geqslant 0$ for every $r=1,\ldots,m-2.$
\end{itemize}

\end{rem}

\section{Model for a class of $m$-isometries}
In this section, we obtain a model for a class of analytic $m$-isometries satisfying \eqref{mu-positive}. We start with a couple of lemmas which will be crucial for the proof of the main theorem of this section.

\begin{lemma}\label{measure construction vector valued}
Let $A$ and $T$ be two operators in $\mathcal B(\mathcal H)$ and $\mathcal E= \ker T^*.$ Suppose $A$ is positive and $T^*AT=A$. Then there exists a $\mathcal B(\mathcal E)$-valued semi-spectral measure $\mu$ on $\mathbb T$ such that  
\begin{align*}
 \langle AT^{l}x,T^{j}y\rangle = \int_{\mathbb T}\zeta^{l-j}d\mu_{x,y}(\zeta),\,\,\,\,l,j\in\mathbb Z_+,\,\,x,y\in\mathcal E.
\end{align*}
\end{lemma}
\begin{proof}
Let $\mathcal A$ be the subspace of $\mathcal H$ given by $\overline{{\rm ran}(A^{1/2})},$ the closure of the range of $A^{1/2}.$ Consider the operator $S$ on ${\rm ran}(A^{1/2})$ defined by $S(A^{1/2}x):= A^{1/2}Tx,$ for $x\in \mathcal H.$ 
Since  $T^*AT=A,$ it follows that $S$ extends to an isometry on $\mathcal A.$ 
By abuse of language, let it be denoted by $S$ itself. 
Suppose $V$ is a unitary extension of $S$ on some Hilbert space $\mathcal K$ containing $\mathcal A$ (see \cite[Proposition I.2.3]{NF}) and $E$ is the spectral measure associated to $V.$ 
Since $V$ is unitary, the support of $E$ is contained in $\mathbb T.$ Let $P_{\mathcal E}$  denote the orthogonal projection of $\mathcal H$ onto $\mathcal E$ and  $P_{\mathcal A}$ denote the orthogonal projection of $\mathcal K$ onto $\mathcal A.$  Now consider the $\mathcal B(\mathcal E)$-valued semi-spectral measure $\mu$ on $\mathbb T$ defined by 
\begin{align*}
\mu(\Delta):= P_{\mathcal E}A^{\frac{1}{2}}P_{\mathcal A}E(\Delta) A^{\frac{1}{2}}|_{\mathcal E},
\end{align*}
for all $\Delta$ in the Borel sigma algebra of $\mathbb T.$ Note that for every $x,y\in\mathcal E,$ we have 
\begin{align*}
\mu_{x,y}(\Delta)= \langle \mu(\Delta)x, y\rangle = \langle E(\Delta)A^{\frac{1}{2}}x, A^{\frac{1}{2}}y\rangle = E_{A^{\frac{1}{2}}x, A^{\frac{1}{2}}y}(\Delta).
\end{align*}
 Thus it follows that 
\begin{align*} 
\langle {V^*}^jV^{l}A^{\frac{1}{2}} x, A^{\frac{1}{2}}y \rangle= \int_{\mathbb T}\zeta^{l-j}d\mu_{x,y}(\zeta),\,\,\,\,l,j\in\mathbb Z_+,\,\,x,y\in\mathcal E.
\end{align*}
Since $V^n A^{\frac{1}{2}} x = S^n A^{\frac{1}{2}} x=  A^{\frac{1}{2}} T^nx,$ for every $x\in\mathcal E$ and $n\in\mathbb Z_+,$ we obtain that 
\begin{align*}
\langle  AT^l x, T^jy \rangle =\langle V^{l}A^{\frac{1}{2}} x, {V}^jA^{\frac{1}{2}}y \rangle= \int_{\mathbb T}\zeta^{l-j}d\mu_{x,y}(\zeta),\,\,\,\,l,j\in\mathbb Z_+,\,\,x,y\in\mathcal E.
\end{align*}
This completes the proof.
\end{proof}
The following lemma shows that the weighted Dirichlet integral  $D_{\mu,n}(\cdot)$ determines the semi-spectral measure $\mu.$  
For a scalar valued polynomial $p$ and $e\in\mathcal E,$ we use the notation $pe$ to denote the $\mathcal E$-valued polynomial given by
$pe(z)=p(z)e$ for every $z\in\mathbb C.$
\begin{lemma}\label{seminorm and measure}
Let $\mu$ and $\nu$ be two $\mathcal B(\mathcal E)$-valued semi-spectral measures on the unit circle $\mathbb T$  and $n\in\mathbb Z_+.$ Suppose $D_{\mu,n}(pe)= D_{\nu,n}(pe)$ for every scalar valued polynomial $p$ and for every $e\in \mathcal E$, then $\mu=\nu.$
\end{lemma}
\begin{proof}
Let $D_{\mu,n}(pe)= D_{\nu,n}(pe)$ for every scalar valued polynomial $p$ and for each $e\in\mathcal E.$ In view of Proposition \ref{general diff. formula vector valued}, we find that $D_{\mu,0}(pe)= D_{\nu,0}(pe)$ for every scalar valued polynomial $p$ and $e\in\mathcal E.$ Now observe that
\begin{align*}
D_{\mu,0}(pe)= \lim_{R\to 1}\displaystyle\int_{\mathbb T} |p(R\zeta)|^2 P_{\!\mu_{e,e}}(R\zeta) d\sigma(\zeta).
\end{align*}
It is well-known that for a positive measure $m\in \mathcal M_+(\mathbb T),$ the measure $P_m(R\zeta)d\sigma(\zeta)$ converges to $m$ in weak*-topology as $R\to 1$ (see \cite[Theorem 3.3.4]{Ru}). Consequently, by continuity of $p$ on $\overline{\mathbb D},$ we obtain that
\begin{align*}
D_{\mu,0}(pe) = \displaystyle\int_{\mathbb T} |p(\zeta)|^2d\mu_{e,e}(\zeta).
\end{align*}
By our assumption along with the polarization identity, we find that for any scalar valued polynomials $p$ and $q,$ 
\begin{align*}
\displaystyle\int_{\mathbb T} p(\zeta) \overline{q(\zeta)}d\mu_{e,e}(\zeta)= \displaystyle\int_{\mathbb T} p(\zeta) \overline{q(\zeta)}d\nu_{e,e}(\zeta).
\end{align*} 
Thus it follows that $\mu_{e,e}=\nu_{e,e},$ that is $\langle\mu(\Delta)e,e\rangle= \langle\nu(\Delta)e,e\rangle$ for every Borel subset $\Delta\subseteq \mathbb T$ and $e\in\mathcal E.$ Hence we  get that $\mu(\Delta)= \nu(\Delta)$ for every Borel subset $\Delta\subseteq \mathbb T.$ This completes the proof.   
\end{proof}

Before we state the main result of this section, in the following proposition, we prove that any $(m-1)$-tuple of $\mathcal B(\mathcal E)$-valued semi-spectral measures $\pmb \mu$ on $\mathbb T$ determines the associated unitary equivalence class of multiplication operator $M_z$ on $\mathcal H_{\pmb\mu}(\mathcal E)$ spaces. This generalizes the result in \cite[Theorem 4.2]{OLLOF2iso} by Olofsson for the case of $m=2.$
\begin{prop}\label{Uniqueness}
Let $m\geqslant 2,$ and $\mathcal E$ and $\mathcal F$ be complex separable Hilbert spaces. Let $\pmb{\mu}=(\mu_1,\ldots,\mu_{m-1})$ and $\pmb{\nu}=(\nu_1,\ldots,\nu_{m-1})$ be two $(m-1)$-tuples of measures in $\mathcal M_+(\mathbb T,\mathcal B(\mathcal E))$ and $\mathcal M_+(\mathbb T,\mathcal B(\mathcal F))$ respectively. Then the operators $M_z$ on $\mathcal H_{\pmb{\mu}}(\mathcal E)$ and $M_z$ on $\mathcal H_{\pmb{\nu}}(\mathcal F)$ are unitarily equivalent if and only if there exists a unitary $V$ in $\mathcal B(\mathcal E,\mathcal F)$ such that $\mu_r(\Delta)=V^*\nu_r(\Delta)V$ for every Borel subset $\Delta \subseteq \mathbb T$ and for all $r=1,\ldots,m-1$.
\end{prop}
\begin{proof}
Let $M_z^{(\pmb{\mu})}$ and $M_z^{(\pmb{\nu})}$ denote the multiplication operators on the Hilbert spaces $\mathcal H_{\pmb{\mu}}(\mathcal E)$ and  $\mathcal H_{\pmb{\nu}}(\mathcal F)$ respectively. 
Suppose that $M_z^{(\pmb{\mu})}$ and $M_z^{(\pmb{\nu})}$ are unitarily equivalent. Then there exists a unitary operator $U:\mathcal H_{\pmb{\mu}}(\mathcal E)\to \mathcal H_{\pmb{\nu}}(\mathcal F)$ such that $UM_z^{(\pmb{\mu})}=M_z^{(\pmb{\nu})}U.$ Since 
$\ker {M_z^{(\pmb{\mu})}}^*= \mathcal E$ and $\ker {M_z^{(\pmb{\nu})}}^* = \mathcal F,$ 
it follows that $U^*(\mathcal F)=\mathcal E$ and consequently we have $U(\mathcal E)=\mathcal F.$ 
Let $V$ be the unitary operator in $\mathcal B(\mathcal E,\mathcal F)$ given by $V= U|_{\mathcal E}.$ Since $UM_z^{(\pmb{\mu})}=M_z^{(\pmb{\nu})}U,$ for any scalar valued polynomial $p$ and $e$ in $\mathcal E,$ $U(pe)=Up(M_z^{(\pmb{\mu})})(e)=p(M_z^{(\pmb{\nu})})(Ue)=p(Ve)
.$
 Also, for $r=1,\ldots,m-1$, 
by a routine verification, we see that 
\begin{align*}
U\Big(\beta_r(M_z^{(\pmb{\mu})})- \sum_{k=1}^{\infty} {L^*}_{\!\!\!\!M_z^{(\pmb{\mu})}}^{k}\beta_{r+1}(M_z^{(\pmb{\mu})}) L_{\!M_z^{(\pmb{\mu})}}^k \Big)=\Big(\beta_r(M_z^{(\pmb{\nu})})- \sum_{k=1}^{\infty} {L^*}_{\!\!\!\!M_z^{(\pmb{\nu})}}^{k}\beta_{r+1}(M_z^{(\pmb{\nu})}) L_{\!M_z^{(\pmb{\nu})}}^k \Big)U.
\end{align*} 
Hence by Theorem \ref{defect formula for model space vector valued}, it follows that 
$D_{\mu_r,0}(f)=D_{\nu_r,0}(Uf)$ for all $f$ in $\mathcal H_{\pmb{\mu}}(\mathcal E)$, and consequently, we obtain that $D_{\mu_r,0}(pe)=D_{\nu_r,0}(p(Ve))$ for any polynomial $p$ and $e\in \mathcal E.$ 
Note that $V^*\nu_r V\in \mathcal M_+(\mathbb T, \mathcal B(\mathcal E))$ and 
\begin{align*}
D_{\nu_r,0}(p(Ve))= D_{V^*\nu_r V,0}(pe),
\end{align*}
for every polynomial $p$ and $r=1,\ldots,m-1.$
Thus in view of Lemma \ref{seminorm and measure}, we conclude that $\mu_r(\Delta)=V^*\nu_r(\Delta)V$ for every Borel subset $\Delta \subseteq \mathbb T$ and for all $r=1,\ldots,m-1.$

For the reverse implication, 
suppose $V$ is a unitary map in $\mathcal B(\mathcal E,\mathcal F)$ such that $\mu_r(\cdot)=V^*\nu_r(\cdot)V$ for $r=1,\ldots,m-1.$ 
Note that $P_{\!\mu_r}(z)=V^*P_{\!\nu_r}(z)V$ for every $z\in\mathbb D$ and for each $r=1,\ldots,m-1.$ The map $V$ induces a linear map $U$ from $\mathcal H_{\pmb{\mu}}(\mathcal E)$ into $\mathcal H_{\pmb{\nu}}(\mathcal F)$ given by
$(Uf)(z)= V(f(z))$ for $f\in \mathcal H_{\pmb{\mu}}(\mathcal E)$ and $z\in\mathbb D.$  It is straightforward to verify that $U$ is unitary from $\mathcal H_{\pmb{\mu}}(\mathcal E)$ onto $\mathcal H_{\pmb{\nu}}(\mathcal F)$ satisfying $UM_z^{(\pmb{\mu})}=M_z^{(\pmb{\nu})}U.$ 
\end{proof}

Now we are ready to prove the main theorem of this section, namely the Theorem \ref{model thm}. This provides a canonical model for the class of analytic $m$-isometries satisfying \eqref{mu-positive} and generalizes \cite[Theorem 5.1]{R} and \cite[Theorem 4.1]{OLLOF2iso}. Note that, if $T$ is an $m$-isometry then the approximate point spectrum $\sigma_{ap}(T)$ is a subset of $\mathbb T$ (see \cite[Lemma 1.21]{AglerStan1}). As $0\notin \sigma_{ap}(T),$ it follows that an $m$-isometry $T$ is always left invertible.
\begin{proof}[\textbf{Proof of Theorem \ref{model thm}}]
The backward implication follows directly from Theorem \ref{bounded m iso vector valued} and Theorem \ref{defect formula for model space vector valued}. 

For the forward implication assume that $T$ is an analytic $m$-isometry satisfying \eqref{mu-positive}. 
From Theorem \ref{Wold-type-mconcave}, we obtain that $T$ has the wandering subspace property. Thus the linear span of $\big\{T^n \mathcal E:n\in\mathbb Z_+\big\}$ is dense in $\mathcal H,$ where $\mathcal E= \ker T^*.$  
Also note that $T^*\beta_{m-1}(T)T= \beta_{m-1}(T).$ From Lemma \ref{m-concaveop}, it follows that $\beta_{m-1}(T) \geqslant 0.$ Moreover, using the relation $L_T T=I$ together with $T^*\beta_r(T)T-\beta_{r+1}(T)= \beta_r(T),$ it follows that for each $r=1,\ldots, m-1,$
\begin{align*}
T^*\Big(\beta_r(T)- \sum_{k=1}^{\infty} {L_T^*}^k\beta_{r+1}(T) L_T^k\Big) T& =\beta_r(T)+\beta_{r+1}(T)- \sum_{k=1}^{\infty} T^*{L_T^*}^k\beta_{r+1}(T) L_T^kT\\
&= \beta_r(T)- \sum_{k=1}^{\infty} {L_T^*}^k\beta_{r+1}(T) L_T^k.
\end{align*} 
For $r=1,\ldots, m-1,$ since the operator $\beta_r(T)- \sum_{k=1}^{\infty} {L_T^*}^k\beta_{r+1}(T) L_T^k$ is positive by hypothesis, as an application of Lemma \ref{measure construction vector valued}, there exists an $(m-1)$-tuple of $\mathcal B(\mathcal E)$-valued semi-spectral measures $\pmb \mu =(\mu_1,\ldots,\mu_{m-1})$  on $\mathbb T$ such that 
\begin{align}\label{difference-identity}
\Big\langle \Big(\beta_r(T)- \sum_{k=1}^{\infty} {L_T^*}^k\beta_{r+1}(T) L_T^k \Big)T^l(x),T^j(y)\Big \rangle = \int _{\mathbb T} \zeta^{l-j}d{(\mu_r)}_{x,y}(\zeta), \quad j,l\in\mathbb Z_+,\quad x,y\in\mathcal E.
\end{align} 
We claim that $T$ is unitarily equivalent to the multiplication operator $M_z$ by the coordinate function on $\mathcal H_{\pmb\mu}(\mathcal E).$
In view of Theorem \ref{defect formula for model space vector valued} together with the polarization identity and \eqref{difference-identity}, we have 
\begin{align}\label{compare of defect}
\langle \beta_{m-1}(M_z) z^lx, z^jy\rangle = \int _{\mathbb T} \zeta^{l-j}d{(\mu_{m-1})}_{x,y}(\zeta) =\langle \beta_{m-1}(T) T^lx, T^jy \rangle,\,\,\,j,l\in\mathbb Z_+,\,\,\, x,y\in\mathcal E.
\end{align}
Let $e\in\mathcal E.$ Since $L_TT=I$ and $L_T(e)=0,$ we get that  
\begin{eqnarray*}
L_T^n T^j(e)=
\begin{cases}
T^{j-n} e & \mbox{ if } n\leqslant j,\\
0 & \mbox{ if } n>j.
\end{cases}
\end{eqnarray*}
In a similar manner, we also have 
\begin{eqnarray*}
L_{M_z}^n z^je=
\begin{cases}
z^{j-n}e & \mbox{ if } n\leqslant j,\\
0 & \mbox{ if } n>j.
\end{cases}
\end{eqnarray*}
This combined with \eqref{compare of defect} gives 
\begin{align}\label{compare}
\Big \langle \beta_{m-1}(M_z) {L^k_{M_z}} z^lx,{L^k_{M_z}}z^jy\Big \rangle = \Big\langle \beta_{m-1}(T) {L^k_T} T^lx,{L^k_T}T^jy 
\Big\rangle,\,\,\,k,l,j\in\mathbb Z_+,\,x,y\in\mathcal E.
\end{align}
Again using Theorem \ref{defect formula for model space vector valued} along with the polarization identity and \eqref{difference-identity}, one obtains
\begin{align*}
\Bigg\langle \Big(\beta_r(M_z)- \sum_{k=1}^{\infty} {L_{M_z}^*}^k\beta_{r+1}(M_z) L_{M_z}^k \Big)z^lx,z^jy\Bigg \rangle
=\Bigg\langle \Big(\beta_r(T)- \sum_{k=1}^{\infty} {L_T^*}^k\beta_{r+1}(T) L_T^k \Big)T^lx,T^jy\Bigg \rangle, 
 \end{align*} 
 for every $j,l\in\mathbb Z_+,$ and $r=1,\ldots,m-2.$ Using this together with \eqref{compare}, inductively, we have 
 \begin{align*}
 \big\langle \beta_{r}(M_z) z^lx, z^jy\big\rangle  =\big\langle \beta_{r}(T) T^lx, T^jy \big\rangle,\,\,\,j,l\in\mathbb Z_+,\,x,y\in\mathcal E,
 \end{align*} for every $r=m-1,\ldots,1.$ From the case of $r=1,$ it follows that 
 \begin{align*}
 \big\langle z^{l+1}x,z^{j+1}y\big\rangle - \big\langle z^{l}x,z^{j}y\big\rangle = \big\langle T^{l+1} x,T^{j+1} y\big\rangle - \big\langle T^{l} x,T^{j} y \big\rangle\,\,\,j,l\in\mathbb Z_+,\,x,y\in\mathcal E. 
 \end{align*}
We also have $\langle x,y\rangle_{\pmb \mu}= \langle x,y\rangle_{\!_{\mathcal H}}$ and $\langle z^lx,y\rangle = \langle T^l x,y\rangle=0$ for every $l\geqslant 1$ and $x,y\in\mathcal E.$ Hence inductively we obtain that 
\begin{align}\label{Grammian match}
\big\langle z^{l}x,z^{j}y\big\rangle =  \big\langle T^{l} x,T^{j} y \big\rangle, \quad j,l\in\mathbb Z_+,\,x,y\in\mathcal E. 
\end{align}
Now consider the map $U$ defined on the linear span of $\big\{T^n \mathcal E:n\in\mathbb Z_+\big\}$ given by 
\begin{align*}
U\Big(\sum\limits_{j=0}^k T^j x_j \Big):=\sum\limits_{j=0}^k z^jx_j, \quad x_0,\ldots,x_k\in\mathcal E.
\end{align*}
From \eqref{Grammian match}, it follows that $U$ is an isometry from the linear span of $\big\{T^n \mathcal E:n\in\mathbb Z_+\big\}$ 
onto the set of all $\mathcal E$-valued polynomials in $\mathcal H_{\pmb\mu}(\mathcal E)$ and the equality $UT=M_zU$ holds on the linear span of $\big\{T^n \mathcal E:n\in\mathbb Z_+\big\}.$  
Since $\mathcal E$-valued polynomials are dense in $\mathcal H_{\pmb\mu}(\mathcal E)$ (by Corollary \ref{poly-dense vector valued}), 
the map $U$ extends as a unitary map from $\mathcal H$ onto $\mathcal H_{\pmb\mu}(\mathcal E)$ satisfying $UT=M_zU.$ 
This completes the proof.
\end{proof}

We conclude this section by showing that the adjoint of every analytic $m$-isometry with $n$ dimensional kernel,  $n\in \mathbb N$, lies in the class $B_n(\mathbb D)$ of  Cowen-Douglas class operators associated to the open unit disc $\mathbb D,$ see \cite{CD} for the definition and other properties. Similar kinds of spectral behaviour for expansive analytic $m$-isometries have been studied in \cite[Lemma 2.6 and Corollary 3.4]{Sameeriso}.

\begin{prop}\label{CD Class}
Let $T \in \mathcal B(\mathcal H)$ be an analytic $m$-isometry with $\dim(\ker T^*)=n,$ for some positive integers $m$ and $n.$ Then it follows that
\begin{itemize}
\item[(i)] $\sigma(T)=\overline{\mathbb D}$ and $\sigma_{ap}(T)=\mathbb T,$
\item[(ii)] $T^*\in B_n(\mathbb D),$
\end{itemize}
where $\sigma (T)$ and $\sigma_{ap}(T)$ denote the spectrum and the approximate point spectrum of the operator $T$ respectively. 
\end{prop}
\begin{proof}
Since $T$ is an $m$-isometry, it follows that $\sigma_{ap}(T)\subseteq \mathbb T$ (see \cite[Lemma 1.21]{AglerStan1}). Thus $T-\lambda I$ is bounded below for every $\lambda \in\mathbb D.$ This gives us that $T-\lambda I$ is semi-Fredholm for every $\lambda \in\mathbb D.$ Since $\ker T^*$ is $n$ dimensional, using continuity of the semi-Fredholm index, we conclude that $\dim \big(\ker(T^*-\lambda I)\big)=n,$ for every $\lambda\in \mathbb D.$ Since the boundary of spectrum of an operator  is contained in the approximate point spectrum (see \cite[Ch.7,6.7]{Con}), it follows that $\sigma(T)=\overline{\mathbb D}$ and $\sigma_{ap}(T)=\mathbb T.$ 
In order to show that $T^*\in B_n(\mathbb D),$ it is sufficient to show that 
\begin{align*}
\bigvee \big\{\ker (T^*-\lambda I): \lambda \in \mathbb D\big\} = \mathcal H.
\end{align*} Note that $(T^*-\lambda I)$ is Fredholm for every $\lambda \in \mathbb D$ with Fredholm index equal to $n.$ 
Now following \cite[Proposition 1.11]{CD}, we obtain that there exist holomorphic $\mathcal H$-valued functions $\{e_i(z): i=1,\ldots,n\}$ defined on some neighborhood $\Omega $ of $0$ such that $\{e_1(z),\ldots,e_n(z)\}$ forms a basis for $\ker(T^*-zI)$ for every $z\in \Omega.$ 
As $T$ is bounded below, it follows that $T^k$ is also bounded below for every $k\in\mathbb N.$ Now from general properties of Fredholm index, see \cite[Ch.11, 3.7]{Con}, we obtain that $\dim(\ker{T^*}^k)=kn,$ for every $k\in\mathbb N.$  Furthermore, from the proof of \cite[Lemma 1.22]{CD}, one may infer that
\begin{align*}
\text{span}\big\{e_1(0),\ldots,e_n(0),\ldots,e_1^{(k-1)}(0),\ldots,e_n^{(k-1)}(0)\big\}= \ker {T^*}^k,\quad k\in\mathbb N.
\end{align*}
Thus, it follows that $\ker{T^*}^k$ is contained in $\bigvee \big\{\ker (T^*-\lambda I): \lambda \in \mathbb D\big\}$ for every $k\in\mathbb N.$ Note that the hyper-range $\mathcal H_{\infty}(T)$ of $T$ is given by
\begin{align*}
\bigg(\bigvee\big\{\ker{T^*}^k:k\in\mathbb N\big\}\bigg)^{\perp}= \mathcal H_{\infty}(T).
\end{align*}
As $T$ is analytic by assumption, we obtain that $\bigvee \{\ker (T^*-\lambda I): \lambda \in \mathbb D\} = \mathcal H.$
\end{proof}

In view of Theorem \ref{bounded m iso vector valued} and Proposition \ref{CD Class} together with  \eqref{cokernel vector valued}, the following corollary is now immediate. The result in \cite[Corollary 3.8(b,c)]{R} can be seen as a special case of the following corollary.
\begin{cor}
Let $\mathcal E$ be a complex Hilbert space of dimension $n$ for some $n\in\mathbb N$ and $\pmb{\mu}=(\mu_1,\ldots,\mu_{m-1})$ be an $(m-1)$-tuple of semi-spectral measures on $\mathbb T.$ The operator $M_z$ on $\mathcal H_{\pmb\mu}(\mathcal E)$ has the following properties:
\begin{itemize}
\item[(i)] $\sigma(M_z)=\overline{\mathbb D}$ and $\sigma_{ap}(M_z)=\mathbb T.$
\item[(ii)] $M_z^*\in B_n(\mathbb D).$
\end{itemize}
\end{cor}

\medskip \textit{Acknowledgement}.
The authors are grateful to Sameer Chavan for his constant support and many  valuable suggestions in the preparation of this article. The authors also wish to express their sincere thanks to Shailesh Trivedi for several comments and fruitful discussions.
We sincerely thank the anonymous referee for several constructive comments which improved the presentation of this article.

\end{document}